\documentclass{amsart}
\usepackage{amsthm,amssymb,amsmath,hyperref, esint}
\usepackage{graphicx}
\usepackage{color}
\usepackage{tikz}
\usetikzlibrary{decorations.pathreplacing}
\usepackage[utf8]{inputenc}
\usepackage{textgreek}
\usepackage{mathtools}

\DeclarePairedDelimiter\floor{\lfloor}{\rfloor}

\definecolor{darkorchid}{rgb}{0.6, 0.2, 0.8}

\theoremstyle{plain}
\newtheorem{thm}{Theorem}[section]

\newtheorem{cor}[thm]{Corollary}

\newtheorem{prop}[thm]{Proposition}

\theoremstyle{definition}
\newtheorem{defn}[thm]{Definition}

\theoremstyle{remark}
\newtheorem{rem}[thm]{Remark}

\theoremstyle{plain}
\newtheorem{conj}[thm]{Conjecture}
\numberwithin{equation}{section}

\newcommand{\eps}{\varepsilon}

\newcommand{\del}{\delta}

\newcommand{\R}{{\mathbb R}}

\newcommand{\N}{{\mathbb N}}

\newcommand{\Z}{{\mathbb Z}}

\newcommand{\calD}{{\mathcal D}}

\def\udot#1{\ifmmode\oalign{$#1$\crcr\hidewidth.\hidewidth
    }\else\oalign{#1\crcr\hidewidth.\hidewidth}\fi}

\def\R{\mathbb{R}}
\def\Z{\mathbb{Z}}

\newtheorem{manualtheoreminner}{Theorem}
\newenvironment{manualtheorem}[1]{%
  \renewcommand\themanualtheoreminner{#1}%
  \manualtheoreminner
}{\endmanualtheoreminner}

\newtheorem{manualpropositioninner}{Proposition}
\newenvironment{manualproposition}[1]{%
  \renewcommand\themanualpropositioninner{#1}%
  \manualpropositioninner
}{\endmanualpropositioninner}

\begin{document}
	
\title[]{Dyadic analysis meets number theory}
\author{Theresa C. Anderson and Bingyang Hu}

\address{Theresa C. Anderson: Department of Mathematics, Purdue University, 150 N. University St., W. Lafayette, IN 47907, U.S.A.}%
\email{tcanderson@purdue.edu}

\address{Bingyang Hu: Department of Mathematics, Purdue University, 150 N. University St., W. Lafayette, IN 47907, U.S.A.}%
\email{hu776@purdue.edu}

\begin{abstract}
We unite two themes in dyadic analysis and number theory by studying an analogue of the failure of the \emph{Hasse principle} in harmonic analysis.  Explicitly, we construct an explicit family of measures on the real line that are $p$-adic doubling for any finite set of primes, yet not doubling, and we apply these results to show analogous statements about the \emph{reverse H\"older} and \emph{Muckenhoupt $A_p$} classes of weights.  The proofs involve a delicate interplay among several geometric and number theoretic properties.
\end{abstract}
\date{\today}

\thanks{The first author is funded by NSF DMS 1954407 in Analysis and Number Theory.  The authors thank Trevor Wooley for important clarifications on the work of Krantz \cite{Krantz} and the prime number theorem, as well as Qiao He, Yuan Liu, and Ruixiang Zhang for helpful discussions and computational evidence surrounding the work in Section \ref{Section number theory}.}

\maketitle

\tableofcontents

\section{Introduction}
Breaking up the real numbers into a union of dyadic pieces is a central technique in analysis.  Dyadic decompositions underscore many major theorems in analysis, as often dyadic pieces are easier to understand and can be treated in different ways.  One can also break up the real numbers into $n$-adic pieces for any $n\in \N$, and the study of such systems has been a frequent topic of investigation.  However, there are still many fundamental unanswered questions pertaining to measures and functions defined on these systems, such as, if a measure is $n$-adic doubling for all $n$, is it doubling?  In this paper we make substantial progress on these questions by uncovering and greatly developing the underlying number theory present when one ``intersects" two different $n$-adic systems.  This reveals an interesting parallel with the failure of the \emph{Hasse principle} in number theory.  Our main results are the following: we construct a family of measures that are $p$-adic doubling for all primes $p$ contained in any finite set, yet not doubling, and we prove several applications to the theory of weights.

The inspiration for this work comes from questions of the form: if $X$ is an operator/object that is in $p$-adic ``class" for every $p$, then is it in ``class" overall?  For example, $X$ could be a measure, or a function, and ``class" could be ``doubling" or $BMO$ (the class of functions of bounded mean oscillation).  For instance, in unpublished work, Peter Jones explored such a question for the ``p-adic" BMO classes (alluded to in \cite{Krantz}, but known amongst analysts).  Answering such a general question has an interesting parallel with the well-known Hasse principle in number theory: that solutions modulo $p$ for every $p$ can be used to create an integer solution.  If one can show that $X$ being in $p$-adic ``class" for each $p$ means it is in ``class", then one has demonstrated a type of Hasse principle in harmonic analysis.  However, since study of the failure of the Hasse principle is a major area of study in number theory and algebraic geometry, showing that $X$ is not in ``class" has many interesting applications as well.

In the authors previous works together (\cite{AHJOW}, \cite{AH}), as well as in other places (\cite{TM}, \cite{LPW}, \cite{A}), a class of numbers called the \emph{far numbers} (specifically ``far from the dyadic rationals") play a large role in understanding \emph{distinct dyadic systems}, which are a set of grids with the property that every cube is contained in a cube from one of the grids of roughly the same size.  Distinct dyadic systems are highly useful (\cite{Conde}, \cite{DCU}, \cite{GJ}, \cite{AL}, \cite{LN}, \cite{P}, \cite{P2} to name just a few), and in our recent works we were able to completely characterize them in both $\R$ and $\R^n$ (\cite{AHJOW}, \cite{AH}).  Essentially, far numbers are bounded away from the dyadic numbers on every scale (see \eqref{defn01} for a precise definition), and by shifting a dyadic grid by a far number, one can create a distinct dyadic systems for small scale cubes.  In a recent paper \cite{BMW}, Boylan, Mills and Ward constructed a concrete example of a measure on $[0, 1]$, which is both dyadic doubling and triadic doubling, but not doubling.  Though seemingly unrelated, since the dyadic and triadic grids do not form a distinct dyadic system, the concept of far numbers was our original inspiration to further the number theory behind the work of Boylan, Mills and Ward.

The bulk of this paper is devoted to generalizing and strengthening \cite{BMW} to construct a family of measures that are both $p$-adic doubling and $q$-adic doubling for primes $q$ and $p$, but not doubling.  That is:
\begin{thm}
\label{main result}
There exists an infinite family of measures that are both $p$-adic and $q$-adic doubling for any distinct primes $p$ and $q$, but not doubling.
\end{thm}

But with just a little more effort, we can actually generalize this even more; indeed, our approach allows us to extend any pairs of primes to any finite collection of primes, which can be viewed a nice complement of Jones's result.

\begin{thm} \label{main result 2}
There exists an infinite family of measures that are $p_i$-adic doubling with $i=1, \dots, M$, for any finite collection of primes $\{p_1, \dots, p_M\}$, but not doubling. 
\end{thm}

While there are some similarities with the set up from \cite{BMW}, both our approach and the necessary number theory and geometry are quite different.  One of our main novelties is expanding the number theory to work in tandem with the underlying geometry in our setup.  In \cite{BMW}, the authors worked with the very specific case of the primes 2 and 3, which allowed not only detailed concrete analysis, but also avoided many technical difficulties that arise with arbitrary primes: for example, they use the fact that 2 is always a primitive root of $3^n, n \ge 1$.  Without such properties, the necessary number theory does not work.  We replace these concepts by a new flavor of number theory depending the the stability of certain orders over specific powers of primes and how they relate to arithmetic progressions.  The basic idea is the following: by using techniques from elementary number and group theory (see Section \ref{Section number theory}), for infinitely many rationals of the form $\frac{k}{p^n}$ (where $k$ lies in a specific arithmetic progression), we are able to find an infinite arithmetic progression of scales $j$ and $m$ such that $\frac{k}{p^n}$ and $\frac{j}{q^m}$ are close in a precise fashion.

We also are able to quantitatively strengthen the geometric results by using the tools described above.  Since we will be looking at intersections of $q$-adic intervals $I$ and $p$-adic intervals $J$, we focus on the proximity of two distinguished points, the left endpoint of the rightmost child of $I$, which we call $\textZeta(I)$, and the right endpoint of the leftmost child of $J$, called $\Upsilon(J)$.  In particular, we are able to show that $\Upsilon$ and $\textZeta$ lie in a certain relative arrangement, and that no matter how small $I$ is, we always have that for a certain $J$, the difference $\Upsilon(J) - \textZeta(I) < \varepsilon |I|$, for any $\eps >0$ that we wish.  While the far numbers had inspired us in exploring this step, they could only take us so far --  in particular, while far numbers allow us to quantify the fact that any $p$-adic interval that intersects an endpoint of a $q$-adic interval must have at least a fixed fraction both inside and outside the interval, this does not give us the quantitative strength ``to be within $\varepsilon$", that we desire.  To visualize this, pick $q = 2, p = 3$.  The interval $[1/3,2/3)$ has a fixed fraction $1/6 = 1/3\cdot 1/2$ to both the left and right of the point $1/2$, and the interval $[0,1/3)$ has at least $1/12 = 1/3\cdot 1/2^2$ to the right and left of the point $1/4$, but what we really want is for the difference $2/3-1/2$ or $1/3-1/4$ to be much smaller on an infinite number of small scales. Therefore, this principle needs to be developed alongside the inherent number theory.  We fully explain our new ingredients, and the new reasoning around them, including the relationship with the far numbers in the article, specifically in Sections \ref{Section number theory} and \ref{geometry}.

In \cite{BMW}, the authors carefully checked that their underlying measure was triadic doubling, including carefully computing the constants, in which consisted the bulk of their article.  Here we completely restructure this part via our ``\emph{exhaustion procedure}", which in particular allows us to treat all nontrivial cases in a unified manner by focusing only on the two rightmost $p$-adic children (described in great detail in Section \ref{exhaustion procedure}).

Finally, we apply our results to the theory of weights.  We begin by showing several facts pertaining to the reverse H\"older weight classes.  Reverse H\"older weights are closely intertwined with the Muckenhoupt $A_p$ weights and are relevant in multiple applications.  Understanding their structure has been a deep topic of investigation (see, for example \cite{CN}).  Via our construction and proof technique for the previous Theorem \ref{main result 2}, we are able to show the following structure theorem for the \emph{prime reverse H\"older classes} (see Section \ref{applications} for a definition):
\begin{thm}
\label{application theorem}
The finite intersection of prime (r)-reverse H\"older classes is never equal to the full (r)-reverse H\"older class.
\end{thm}
This theorem holds true for all $1 \leq r <\infty$, and the proof requires a careful selection of the $r$ parameter governing the reverse H\"older classes, as well as an elaborated analysis of the structural properties of the measure used to satisfy Theorems \ref{main result} and \ref{main result 2}.  This leads us to prove a similar statement for \emph{Muckenhoupt $A_p$ weights}, which leads to the following (see Section \ref{applications} for a definition):
\begin{thm}
\label{Ap thm}
The Muckenhoupt class $A_\infty$ is not equal to the intersection of finitely many prime $A_\infty$ classes.
\end{thm}

More details are found in Section \ref{applications}.  These provide more analogues to Jones's result on the BMO classes; likely many other related applications are possible.
  
One may wonder if our results are extendable to general integers $n_1$ and $n_2$ instead of primes.  First note that if $n_1 = n_2^k$, then it can be easily shown that any $n_1$-adic doubling measure is automatically $n_2$-adic doubling (and vice versa) with constants $C$ and $C^k$.  Hence one could ask if we can get the same results outside of this situation, or even in the case $(n_1, n_2) = 1$.  Even in the latter case this appears to be a very difficult question.  While the construction of the measure and the analysis employed in Sections \ref{analysis 1} - \ref{padic3} could carry through in this setting, we would still crucially rely on the underlying number theory connected to the geometry of this setting, where it appears that several new ideas would be needed (for example, to name one: in place of Fermat's little theorem one can naively expect to use Euler's theorem, but this creates several bottlenecks).  Other, similarly more difficult, extensions are possible; it appears that this area has a lot of intriguing possibilities to explore. 

The organization of our paper is the following: Section \ref{Section number theory} elaborates the number theory that we develop and how it connects to our underlying geometry, which is further detailed in Section \ref{geometry}.  The analytic aspects begin in Section \ref{analysis 1}, with many visible connections to the work \cite{BMW}, but the bulk of our work and novelty occur in Sections \ref{padic1}, \ref{padic2} and \ref{padic3} where we show that our measure is $p$-adic doubling.  In particular, Section \ref{padic1} contains Figures of our measure, and a description, both in text and mathematically, of our exhaustion procedure.  Section \ref{padic2} works out the math for the most involved cases, and Section \ref{padic3} finishes the work, as well as derives Theorem \ref{main result 2} by generalizing the underlying number theory and geometry a bit to handle the new complexity.   Finally, applications to the theory of weights appear in Section \ref{applications}. 

\bigskip

\section{The number theory part}
\label{Section number theory}

The goal of this section is to deal with the number theory which plays an important role in constructing the desired measure in Theorem \ref{main result}. In particular, our results generalize the number theory results which was used in \cite{BMW} in a non-trivial way. More precisely, the authors in \cite{BMW}  consider the case when $(p, q)=(3, 2)$ and their results heavily depended on the following three facts:
\begin{enumerate}
    \item [(a).] $2$ is a primitive root of $3^n$ for any $n \ge 1$;
    
    \medskip
    
    \item [(b).] $2^{3^{n-1}} \equiv 3^n-1 \ (\textrm{mod} \ 3^n)$ for $n \ge 1$;
    
    \medskip
    
    \item [(c).] $2^{2 \cdot 3^{n-2}} \equiv 3^{n-1}+1 \ (\textrm{mod} \ 3^n)$ for $n \ge 2$. 
\end{enumerate}
Note that these facts can easily fail for other pairs of primes. For example, $(p, q)=(7, 2)$, since $2$ is not a primitive root of $7$, hence also for $7^n$ for any $n \ge 1$.  Our result (see, Proposition \ref{120200724prop01}) overcomes these difficulties and allows us to extend the result in \cite{BMW} to any pairs of primes. 

Recall that $p$ and $q$ are two distinct primes. Without the loss of generality, we may assume $p>q$. To begin with, by Fermat's little theorem, we know that
$$
p^{q-1} \equiv 1 \quad (\textrm{mod} \ q)
$$
and 
$$
q^{p-1} \equiv 1 \quad (\textrm{mod} \ p). 
$$
Moreover, we denote $(\Z / n\Z)^*$ the multiplicative group of integers modulo $n$, $n \in \N$.

\begin{prop} \label{120200724prop01}
Let $p, q$ be two distinct primes. Let further, $O_m(p, q)$ be the order of $q^{p-1}$ in $\left( \Z \big/ \left(p^m \Z \right) \right)^*$ for each $m \ge 1$. Then there exists some integer $C(p, q)\geq 0$, such that
$$
 \frac{O_m(p, q)}{p^{m-1}}=\frac{1}{p^{C(p, q)}}. 
$$
when $m$ is sufficiently large. 
\end{prop}
\begin{rem}
Note that the claim is easily true for $C(p,q,m) \geq 0$ (where the constant is allowed to depend on $m$) since $O_m(p, q) \mid p^{m-1}$ by Euler's theorem.  The nontrivial claim is that for all $m$ large enough, this constant does not depend on $m$.  Indeed, not only will we be able to ``cancel" the tern $p^m$ in the denominator, but what is left over will be both independent of $m$ and unchanging for $m$ large.
\end{rem}
\begin{proof}
Let $m(p, q)$ be the smallest integer such that
$$
q^{p-1} \not\equiv 1 \ (\textrm{mod} \ p^{m(p, q)+1}). 
$$
This implies that there exists some $N_0 \in \{1, 2, \dots, m(p, q)\}$ such that
\begin{equation} \label{20200726eq01}
\left(q^{p-1} \right)^{p^{N_0}} \equiv 1 \ (\textrm{mod} \ p^{m(p, q)+1}), 
\end{equation}

since by Euler's theorem (applied to $q$), it is always true that
$$
\left(q^{p-1} \right)^{p^{m(p, q)}} \equiv 1 \ (\textrm{mod} \ p^{m(p, q)+1}).
$$
Without the loss of generality, we assume that the $N_0$ fixed above is the smallest, namely
\begin{equation} \label{20200726eq02}
\left(q^{p-1} \right)^{p^{N_0-1}} \not\equiv 1 \ (\textrm{mod} \ p^{m(p, q)+1}). 
\end{equation}

\medskip

\textbf{Claim:} For any $\ell \ge 0$, there holds
\begin{equation} \label{20200727eq01}
    \left(q^{p-1} \right)^{p^{N_0+\ell-1}} \not\equiv 1 \ (\textrm{mod} \ p^{m(p, q)+\ell+1}). 
\end{equation}

\medskip

We prove the claim by induction. The case when $\ell=0$ is exactly \eqref{20200726eq02}. Assume \eqref{20200727eq01} holds when $\ell=k$, that is
\begin{equation} \label{20200727eq02}
\left(q^{p-1} \right)^{p^{N_0+k-1}} \not\equiv 1 \ (\textrm{mod} \ p^{m(p, q)+k+1}), 
\end{equation} 
and we have to prove it for the case when $\ell=k+1$.  By \eqref{20200726eq01} and the fact that: if
$$
a \equiv b \ (\textrm{mod} \ p^\ell),
$$
then 
$$
a^p \equiv b^p \  (\textrm{mod}  \ p^{\ell+1}).
$$
we have
\begin{equation} \label{20200727eq02'}
\left(q^{p-1} \right)^{p^{N_0+k-1}} \equiv 1 \ (\textrm{mod} \ p^{m(p, q)+k}), 
\end{equation}
which, together with \eqref{20200727eq02}, implies that we can write
$$
\left(q^{p-1} \right)^{p^{N_0+k-1}}=p^{m(p, q)+k} \cdot s+1, 
$$
where $p \nmid s$ (otherwise it contradicts \eqref{20200727eq02}). 

Taking the $p$-th power on both sides of the above equation, we have
\begin{eqnarray*}
\left(q^{p-1} \right)^{p^{N_0+k}}%
&\equiv& \left(\left(q^{p-1} \right)^{p^{N_0+k-1}} \right)^p \\
&\equiv& \left(p^{m(p, q)+k} \cdot s+1 \right)^p \\
&\equiv& p^{m(p, q)+k+1} \cdot s+1 \\
& \not\equiv& 1 \quad (\textrm{mod} \ p^{m(p, q)+k+2}); 
\end{eqnarray*}
in the last line above, we use the fact that $p \nmid s$. Therefore, \eqref{20200727eq01} is proved. 

\medskip

Now from \eqref{20200727eq01} and \eqref{20200727eq02'} with $k=\ell+1$, we conclude that there exists some $N_0 \in \{1, \dots, m(p,q)\}$, such that for each $\ell \ge 0$, 
$$
O_{m(p, q)+\ell+1}(p, q)=p^{N_0+\ell},
$$
which  by setting $\ell = m-m(p,q) -1$, implies that when $m \ge m(p, q)+2$, the ratio
$$
\frac{O_m(p, q)}{p^{m-1}}= \frac{O_{m(p,q)+\left(m-m(p,q)-1\right)+1}(p, q)}{p^{m-1}} = \frac{p^{m-m(p, q)-1+N_0}}{p^{m-1}}=\frac{1}{p^{m(p, q)-N_0}}
$$
stabilizes, with $C(p, q)=m(p, q)-N_0$. 
\end{proof}

This statement along with the next fact will completely generalize the underlying number theory and geometry in \cite{BMW}.  After the following proof we fully explain the connection.

\begin{prop}
\label{Prop arithmetic progressions}
Let $p$, $q$ be two distinct primes and $C(p, q)$, $m(p, q)$ be defined as in Proposition \ref{120200724prop01}. Then for any $m_1>\frac{m(p, q)}{q-1}$ and 
\begin{eqnarray} \label{20200727eq04}
k %
&\in& \left\{1, 1+p^{C(p, q)+1}, 1+2p^{C(p, q)+1}, \dots, p^{m_1(q-1)}-p^{C(p, q)+1}+1 \right\} \nonumber \\
&=& \left\{ a \in \left[1, p^{m_1(q-1)}\right]: a \equiv 1 \  \left(\textrm{mod} \ p^{C(p, q)+1} \right) \right\}, 
\end{eqnarray}
there exists infinitely many pairs $j$ and $m_2$, where $m_2 \in \N$, and 
\begin{eqnarray} \label{20200727eq05}
j%
&\in&  \left\{q-1, 2q-1, \dots, q^{m_2(p-1)}-1 \right\} \nonumber \\
&=& \left\{b \in \left[ 1, q^{m_2(p-1)} \right]: b \equiv -1 \ (\textrm{mod} \ q) \right\}, 
\end{eqnarray}
such that
\begin{equation} \label{20200727eq06}
\frac{k}{p^{m_1(q-1)}}-\frac{j}{q^{m_2(p-1)}}=\frac{1}{p^{m_1(q-1)}q^{m_2(p-1)}}.
\end{equation} 
\end{prop}

\begin{proof}
It is clear that \eqref{20200727eq06} is equivalent to find infinitely many pairs $m_2$ and $j$ which satisfies \eqref{20200727eq05} for the equation 
\begin{equation} \label{20200727eq07}
k q^{m_2(p-1)}-j p^{m_1(q-1)}=1, 
\end{equation}
where $m_1> \frac{m(p, q)}{q-1}$ and $k$ satisfies \eqref{20200727eq04}. 

To begin with, we note that if \eqref{20200727eq07} holds, then $j$ automatically satisfies \eqref{20200727eq06}. Indeed, this follows easily by taking the modulus $q$ on both sides of \eqref{20200727eq07} and the Fermat's little theorem. 
 
Therefore, it suffices for us to solve \eqref{20200727eq07} for infinitely many pairs $m_2$ and $j$. Taking modulus $p^{m_1(q-1)}$ on both sides of \eqref{20200727eq07}, we see that it suffices to solve 
\begin{equation} \label{20200727eq08}
    kq^{m_2(p-1)} \equiv 1 \quad \left(\textrm{mod} \ p^{m_1(q-1)}\right),
\end{equation}
where 
$$
k \in \left\{ a \in \left[1, p^{m_1(q-1)}\right]: a \equiv 1 \  \left(\textrm{mod} \ p^{C(p, q)+1}\right) \right\}.
$$
Denote
$$
G_{m_1}(p, q):=\left\{ a \in \left[1, p^{m_1(q-1)}\right]: a \equiv 1 \  \left(\textrm{mod} \ p^{C(p, q)+1}\right) \right\}.
$$
The solubility of \eqref{20200727eq08} will follow from the following facts.
\begin{enumerate}
    \item [(a).] The set $G_{m_1}(p, q)$ is a subgroup of $\left(\Z / \left(p^{m_1(q-1)} \right) \Z \right)^*$; 
    
    \medskip
    
    \item [(b).] $q^{p-1}$ is a generator of the group $ G_{m_1}(p, q)$. 
\end{enumerate}

Suppose both $(a)$ and $(b)$ hold, it follows that there exists some $m' \in \N$, such that
$$
k \equiv q^{m'(p-1)} \  \left(\textrm{mod} \ p^{m_1(q-1)}\right), 
$$
which implies the desired assertion.  Indeed, this is because
$$
k \in G_{m_1}(p, q)= \langle q^{p-1} \rangle \subseteq \left(\Z / \left(p^{m_1(q-1)} \right) \Z \right)^*. 
$$

Hence, 
$$
kq^{m_2(q-1)} \equiv q^{m'(p-1)} \cdot q^{m_2(p-1)} \equiv q^{(m'+m_2)(p-1)} \left(\textrm{mod} \ p^{ m_1(q-1)}\right).
$$
Now we wish to find $m_2$ such that
$$
q^{(m'+m_2)(p-1)} \equiv 1 \ \left(\textrm{mod} \ p^{ m_1(q-1)}\right)
$$
and this is guaranteed by the fact that
$$
q^{p-1} \quad \textrm{and} \quad p^{ m_1(q-1)}
$$
are coprime and Euler's theorem.    

We now show (a) and (b).  To begin with, we note that
\begin{equation} \label{20200727eq08}
q^{p-1} \equiv 1 \left(\textrm{mod} \ p^{C(p, q)+1}\right), 
\end{equation}
that is, $q^{p-1} \in G_{m_1}(p, q)$. This is because $m(p, q)$ is the smallest integer such that
$$
q^{p-1} \not\equiv 1 \ (\textrm{mod} \ p^{m(p, q)+1})
$$
and $C(p, q)=m(p, q)-N_0$ for some $N_0 \in \{1, \dots, m(p, q)\}$. Next, as an easy consequence of \eqref{20200727eq08}, it is also easy to see that
$$
\left(q^{p-1} \right)^\ell \in G_{m_1}(p, q), \quad \forall \ell \ge 1, 
$$
and hence
$$
\left\langle q^{p-1} \right\rangle \subseteq G_{m_1}(p, q), 
$$
where $\langle q^{p-1} \rangle$ is the cyclic group generated by $q^{p-1} \in \left(\Z / \left(p^{m_1(q-1)} \right) \Z \right)^*$.

Therefore, both assertions $(a)$ and $(b)$ will follow if we can actually show that
$$
\left\langle q^{p-1} \right\rangle=G_{m_1}(p, q),
$$
which follows from the fact that
$$
O_{m_1(q-1)}=p^{m_1(q-1)-C(p, q)-1}= \left|G_{m_1}(p, q) \right|
$$
which crucially hinges on Proposition \ref{120200724prop01}.
The proof is complete. 
\end{proof}

\begin{rem}
\label{going beyond far numbers}
Propositions \ref{120200724prop01} and \ref{Prop arithmetic progressions} indeed can be viewed as a generalization of \cite[Claim 1.13]{BMW}, in which, the authors considered the case when $p=3$ and $q=2$, and they showed that:

\emph{For any $n \in \N$ and $k \in \{1, 4, 7, \dots, 3^n-2\}$, there exists infinitely many pairs $j$ and
 $m$, where $m \in \N$ and $j \in \{1, 3, 5, \dots, 2^m-1\}$, such that
\begin{equation} \label{20200727eq03}
\frac{k}{3^n}-\frac{j}{2^m}=\frac{1}{2^m 3^n}.
\end{equation}
}

A thorough description is in order.

If one wants to generalize \eqref{20200727eq03} to primes $p,q$, the most obvious approach is to show for all $n$ large enough and $k \in \{1,p+1,2p+1, \dots \}$ that there exist infinitely many $m$ and $j\in \{1,q-1, 2q-1, \dots \}$ (one might also think instead of claiming $j \in \{1,q+1, 2q+1, \dots \}$, but this does not work and is not suited to the geometry of the problem) such that 
\begin{equation}\label{20200727eq03'}
    \frac{k}{p^n}-\frac{j}{q^m}=\frac{1}{q^m p^n}.
\end{equation}
A natural choice would be to use the properties of \emph{far numbers} that we have developed in \cite{A}, \cite{AHJOW} (see Definition \ref{defn01}).  Specifically, following similar ideas in all of these (see for example Lemma 2 and Remark 4 in \cite{A}), one can show that 
\[
 \left|\frac{k}{p^n}-\frac{j}{q^m}\right|\geq \frac{C_{p,n}}{q^m}
 \]
with $C_{p,n} = 1/p^n$ and that there exists an infinite sequence of $m$ and $j$ where this $C$ is the best possible.  We start with the first claim: plugging in $C = 1/p^n$ means that we must show that $|kq^m-jp^n| \geq 1$, that is, to prevent $kq^m = jp^n$.  But this is easy since we assume that $q\nmid j$ and $p\nmid k$.  For the second claim, let $C = 1/p^n+\varepsilon$ for $\varepsilon>0$ and arbitrarily small.  Then if this constant were to work, we would need to prevent $kq^m = jp^n+1$ and $kq^m = jp^n-1$.  However, since 
\[
q^{\varphi(p^n)} \equiv 1 \ \left(\textrm{mod} \ p^n\right)
\]
we have, if $k=1$, that $C = 1/p^n$ precisely for an infinite sequence of $m$, $j$, so we have equality (up to possible negative signs) in \eqref{20200727eq03'}.  Moreover, one can easily see by taking moduli that $j \equiv -1  \ \left(\textrm{mod} \ q\right)$.

There is one major issue with this argument: a priori it only works for $k=1$.  This is where Proposition \ref{120200724prop01} and \ref{Prop arithmetic progressions} come into play: the former allows the fractions $k/p^n$ to stabilize when paired with the latter's more specific arithmetic progression of $k$.  That is, by restricting $k$ to lie in the sequence $1 \mod p^h$, for a specific choice of $h$ independent of $n$ (our \eqref{20200727eq04}), we essentially replicate the desired situation where $k$ is always equal to 1 always.

Some additional commentary relating our generalization compared with \cite{BMW} includes:

\begin{enumerate}
    \item [(1).] From the view of number theory, since $2$ is a primitive root modulo $3^n$, the order of $2$ modulo $3^n$ is $2 \cdot 3^{n-1}$, and hence $O_n(3, 2)=3^{n-1}$ always; moreover the set $\{1, 4, 7, \dots, 3^n-2\}$ is exactly the subgroup of $1\mod 3$ contained in $(\Z/3^n\Z)^*$.  Note that $q^{p-1} = 4$, which is generator of this subgroup, mirroring our proof of \eqref{20200727eq08}.  All these suggest that one probably can restate \cite[Claim 1.13]{BMW} simply by
    $$
    C(3, 2)=0.
    $$
    using our work.
    \medskip
    
    \item [(2).] From the view of analysis, we first note that one does not need the full strength of the condition  ``for any $n \in \N$" in the statement of \cite[Claim 1.13]{BMW} for its application to constructing the desired measure. Indeed, it suffices if the statement of \cite[Claim 1.13]{BMW} holds for a subsequence $\{n_i\}_{i \ge 1}$. This is clear from the proof of \cite[Theorem 1.12]{BMW} (see, \cite[Page 272, Line 19--22]{BMW}). In Proposition \ref{120200724prop01}, we have shown that the ratio $\frac{O_m(p, q)}{p^{m-1}}$ stabilizes for $m$ sufficiently large (which is $n$ in \cite{BMW}, which is stronger than necessary. 
    
    Moreover, Proposition \ref{120200724prop01} also generalizes the geometric structure inherited in \cite[Claim 1.13]{BMW}. More precisely, Propositions \ref{120200724prop01} and \ref{Prop arithmetic progressions} generalize the claim that there exists a pair of intervals $I$ and $J$, such that 
    $$
    \Upsilon(J)-\textZeta(I)=\frac{1}{2^m 3^n},
    $$
    where $J$ is a tri-adic interval of sidelength $\frac{1}{3^{n-1}}$ and $I$ is a dyadic interval of sidelength $\frac{1}{2^{m-1}}$ with $m$ even, for any $p$ and $q$.  This will be highlighted in future sections. 
\end{enumerate}
\end{rem}

\subsection{Relationship to a conjecture of Krantz}
There is a connection to a conjecture of Krantz \cite{Krantz} alluded to in \cite{BMW}, and we believe that some clarifications on this would be helpful. The conjecture of Krantz \cite{Krantz} is the following. 

\begin{conj} [\cite{Krantz}] \label{K3}
Let $\epsilon>0$ be arbitrarily small, and $m \in \N$ be sufficiently large. For each positive integer $k$, does there exists a prime $p$, an integer $\beta$ and an integer $n$ with $1/10 \le p^n/2^m \le 10$, such that
$$
\left| \frac{k}{2^m}-\frac{\beta}{p^n} \right|<\frac{\epsilon}{2^m}?
$$
\end{conj}

Krantz then suggests that if the prime numbers were ``strongly randomly distributed", then this conjecture would be true.  This is not exactly the case; what Krantz mentions would be true if the $m$ in the conjecture is allowed to vary (that is, $m$ can depend on $\epsilon$).  More precisely, if we replace the fixed $m$ (which is $n$ in Krantz's original notation) in the conjecture above by allowing $m$ (and therefore $p$) to vary, then the strong randomness property implies that this (new) statement is true.  The strong randomness property is indeed true via the strong form of the prime number theorem (prime number theorem with error term), which is a classical result.  Therefore, we will interpret Krantz's conjecture in the manner that he writes, but emphasize that the strong randomness property is true and that this property implies a slightly different statement than what Krantz wrote.  This is a subtle, yet important difference.   

While for Conjecture \ref{K3}, we can disprove it by showing the following result: 

\begin{prop} \label{failK3}
For any $m \in \N$, there exists some constant $C(m)>0$, which only depends on $m$, such that
$$
\inf_{\beta \in \N, \  n, p \ \textrm{prime}: \frac{1}{10} \le \frac{p^n}{2^m} \le 10} \left| \frac{1}{2^m}-\frac{\beta}{p^n} \right| \ge \frac{C(m)}{2^m}. 
$$
\end{prop}

Assuming Proposition \ref{failK3}, we turn to disprove Conjecture \ref{failK3}. To do this, we fix a $m$ sufficiently large, and let $\epsilon=\frac{C(m)}{2}$. Moreover, we let $k=1$ there,  which is why we have $\frac{1}{2^m}$ in the statement of Proposition \ref{failK3}, instead of $\frac{k}{2^m}$. The contradiction follows immediately. 

Proposition \ref{failK3} is an immediate consequence of our previous work \cite{AHJOW}. To begin with,  we recall the definition of \emph{far numbers}. 
\begin{defn} \label{defn01}
A real number $\del$ is \emph{$n$-far} if the distance from $\del$ to each given rational $k/n^m$ is at least some fixed multiple of $1/n^m$, where $m \geq 0$, $k\in \Z$. That is, if there exists $C>0$  such that
\begin{equation}
\label{C delta}
   \left| \del-\frac{k}{n^m} \right| \ge \frac{C}{n^m}, \quad \forall m \ge 0, k \in \Z.
\end{equation}
where $C$ may depend on $\del$ but independent of $m$ and $k$. 
\end{defn}

\begin{proof} [Proof of Proposition \ref{failK3}.]
Let us fix $m \in \N$. First of all, since
\begin{equation} \label{20200907eq01}
\frac{1}{10} \le \frac{p^n}{2^m} \le 10, 
\end{equation} 
it is clear that there are only finitely many pairs $(p, n)$ which satisfies \eqref{20200907eq01}, and such a number only depends on $m$,

Now let us take one of these pairs and denote it as $(p, n)$. Observe that there are only finitely many non-zero $\beta$ such that the following holds:
\begin{equation} \label{20200907eq02}
\left|\frac{1}{2^m}-\frac{\beta}{p^n} \right| \le 1, 
\end{equation}
and such a number of $\beta$'s only depends on $p$ and $n$, and hence only depends on $m$. While for those $\beta$ that fail the estimate and $\beta=0$, it suffices, for example, to take $C(m)=1$. To this end, we take a non-zero $\beta$ which satisfies \eqref{20200907eq02}.

Now we apply \cite[Corollary 2.10, (a)]{AHJOW} to see that there exists a constant $C_{\beta, p, n}>0$, which only depends on $\beta, p$ and $n$, such that
$$
\left| \frac{1}{2^{\widetilde{m}}}-\frac{\beta}{p^n} \right| \ge \frac{C_{\beta, p, n}}{2^{\widetilde{m}}}, \quad \forall \widetilde{m} \ge 1,
$$
since $\frac{\beta}{p^n}$ is far with respect to $2$. In particular, letting $\widetilde{m}=m$, we have
$$
\left| \frac{1}{2^m}-\frac{\beta}{p^n} \right| \ge \frac{C_{\beta, p, n}}{2^m}. 
$$
Finally, it suffices to take 
$$
C(m):=\min \left\{\min_{\substack{ \beta: \beta \ \textrm{satisfies \eqref{20200907eq01}}, \beta \neq 0, \\  n, p \ \textrm{prime}: \frac{1}{10} \le  \frac{p^n}{2^m} \le 10}} C_{\beta, p, n}, 1\right\},
$$
which  only depends on $m$.
\end{proof}

\section{The selection procedure}
\label{geometry}
The purpose of this section is to select a collection of disjoint $q$-adic intervals, and these intervals can be treated as a building block of the example of the desired measure. These intervals are chosen carefully according to Proposition \ref{120200724prop01} and Proposition \ref{Prop arithmetic progressions}. 

We recall several definitions.
\begin{defn} \label{20200809defn01}
A \emph{doubling measure} $\mu$ is a measure for which there exists a positive constant $C$ such that for every interval $I \subset \R$, $\mu(2I) \le C\mu(I)$, where $2I$ is the interval which shares the same midpoint of $I$ and twice the length of $I$. 
\end{defn}

\begin{defn}
For $n \ge 1, n \in \N$, the \emph{standard $n$-adic system} $\calD(n)$ is the collection of $n$-adic intervals in $\R$ of the form
\begin{equation} \label{20200722eq01}
I=\left[ \frac{k-1}{n^m}, \frac{k}{n^m} \right), \quad m, k \in \Z. 
\end{equation}
The $n$-adic children of the interval defined in \eqref{20200722eq01} are 
\begin{equation} \label{20200722eq02}
I_j=\left[ \frac{k-1}{n^m}+\frac{j-1}{n^{m+1}}, \frac{k-1}{n^m}+\frac{j}{n^{m+1}} \right), \quad 1 \le j \le n. 
\end{equation}
Moreover, we write $\Upsilon(I)$ be the right endpoint of $I_1$, that is
\begin{equation} \label{20200724eq02}
\Upsilon(I_1)=\frac{k-1}{n^m}+\frac{1}{n^{m+1}},
\end{equation} 
and $\textZeta(I)$ be the left endpoint of $I_n$, that is 
\begin{equation} \label{20200724eq03}
\textZeta(I_n)=\frac{k-1}{n^m}+\frac{n-1}{n^{m+1}}. 
\end{equation}
Finally, we denote $l(I)$ the left endpoint of $I$ and $r(I)$ the right endpoint of $I$ as usual. 
\end{defn}

\begin{defn}
A measure $\mu$ is a \emph{$n$-adic doubling measure} if there exists a positive constant $C$, independent of all parameters, such that for any $n$-adic interval $I$ of the form \eqref{20200722eq01}, 
$$
\frac{1}{C} \le \frac{\mu(I_{j_1})}{\mu(I_{j_2})} \le C, 
$$
where both $I_{j_1}$ and $I_{j_2}$ are some $n$-adic children of $I$, which take the form \eqref{20200722eq02}. The smallest possible constant $C$ are called the \emph{$n$-adic doubling constant} of $\mu$. 
\end{defn}

Recall that we assume that $p$ and $q$ are two distinct primes, with $p>q$. 

\begin{thm} \label{20200722thm01}
There exists a collection of $q$-adic intervals $\{I^{\alpha_\ell}_\ell \}_{\ell \ge 1}$ on $[0, 1)$, where $\alpha_\ell \ge 1$ is a positive integer associated to $\ell$, such that
\begin{enumerate}
    \item [(1).] The collection of $p$-adic intervals $\{J^\ell\}_{\ell \ge 1}$ is pairwise disjoint and contained in $[0, 1)$, where $J^\ell$ is the smallest $p$-adic interval that contains $I_\ell^{\alpha_\ell}$. In particular, the collection $\{I_\ell^{\alpha_\ell}\}_{\ell \ge 1}$ is also pairwise disjoint;
    
    \medskip

    \item [(2).] For each $\alpha \ge 1, \alpha \in \N$, there are only finitely many $\ell \ge 1$, such that $\alpha_\ell=\alpha$. 
    
    \medskip
    
    \item [(3).] For each $\ell \ge 1$, 
    \begin{equation} \label{20200722eq03}
    0<\Upsilon\left(J^\ell \right)-\textZeta\left(I_\ell^{\alpha_\ell} \right) \le q^{-100\alpha_\ell} \left| I_\ell^{\alpha_\ell} \right|. 
    \end{equation}
    Note that since $J^\ell$ is the smallest $p$-adic interval which contains $I_\ell^{\alpha_\ell}$, condition \eqref{20200722eq03} in particular guarantees that the right endpoint of $I_\ell^{\alpha_\ell}$ is to the right of $\Upsilon(J^\ell)$;

\end{enumerate}
\end{thm}

The specific role of $\alpha$ will be detailed later on; basically it represents the number of generations that we will alter to construct our measure.  We'll refer to the points $\Upsilon$ and $\textZeta$ as distinguished points.

To prove this result, we need the following proposition.

\begin{prop} \label{20200805prop01}
Given any interval $\widetilde{J} \subset [0, 1]$ ($\widetilde{J}$ is not necessarily $p$-adic) and any $\varepsilon>0$, there exists a $q$-adic interval $I \subset \widetilde{J}$ such that 
$$
0<\Upsilon(J)-\textZeta(I) \le \varepsilon|I|,
$$
where $J$ is the smallest $p$-adic interval that contains $I$. 
\end{prop}

\begin{rem}
One of the key differences between our approach and \cite{BMW} is that we can make the difference between the distinguished points arbitrarily small.  This simplifies some of the analysis and allows for great flexibility in our construction.
\end{rem}
\begin{proof}
We start with fixing an interval $\widetilde{J} \subset [0, 1]$ and some $\varepsilon>0$, and we let $J'$ be the largest $q$-adic interval which is contained in $\widetilde{J}$ with sidelength $\frac{1}{q^{m_1'}}$. We choose $m_1> \max\left\{\frac{m(p, q)}{q-1},  m_1' \right\}$ and $\frac{1}{p^{m_1(q-1)}}<\varepsilon q$, and we choose
$$
k \in \left\{1, 1+p^{C(p,q)+1}, 1+2p^{C(p, q)+1}, \dots, p^{m_1(q-1)}-p^{C(p, q)+1}+1 \right\}, 
$$
such that $\frac{k}{p^{m_1(q-1)}} \in J'$. Fix such a pair of $m_1$ and $k$, and let 
$$
J:=\left[ \frac{k-1}{p^{m_1(q-1)}}, \frac{k+p-1}{p^{m_1(q-1)}} \right].
$$
Note that we then have
$$
\Upsilon(J)=\frac{k}{p^{m_1(q-1)}}
$$
and $J \subseteq J' \subseteq \widetilde{J}$ due to the choice of $m_1$ and the fact that $p>q$. 

By Proposition \ref{Prop arithmetic progressions}, there exists infinitely many pairs $m_2 \in \N$ and 
$$
j \in \left\{q-1, 2q-1, \dots, q^{m_2(p-1)}-1 \right\},
$$
such that
\begin{equation} \label{20200805eq01}
\frac{k}{p^{m_1(q-1)}}-\frac{j}{q^{m_2(p-1)}}=\frac{1}{p^{m_1(q-1)}q^{m_2(p-1)}}.
\end{equation} 
We choose such a pair $m_2$ and $j$, with $m_2$ sufficiently large such that
\begin{equation} \label{20200805eq02}
q^{m_2(p-1)}>10q \cdot p^{m_1(q-1)}
\end{equation} 
and let
$$
I:=\left[\frac{j+1-q}{q^{m_2(p-1)}}, \frac{j+1}{q^{m_2(p-1)}} \right].
$$
Note that 
$$
\textZeta(I)=\frac{j}{q^{m_2(p-1)}}.
$$
The desired result will follow if the following assertions are verified: 
\begin{enumerate}
    \item [(i).] $\Upsilon(J)>\textZeta(I)$;
    \item [(ii).] $I \subset J$;
    \item [(iii).] $J$ is the smallest $p$-adic interval containing $I$; 
    \item [(iv).] $\Upsilon(J)-\textZeta(I)<\varepsilon |I|$.
\end{enumerate}

\medskip

\textit{Proof of (i).} This is clear by \eqref{20200805eq01}.

\medskip

\textit{Proof of (ii).} This is indeed guaranteed by \eqref{20200805eq02}. More precisely, let $i(I)$ be the left endpoint of $I$ and $d(I)$ be the right endpont.  Then we have
\begin{eqnarray*}
\left|[i(I), \Upsilon(J)]\right|%
&=& \left|[i(I), \textZeta(I)]\right|+\left|[\textZeta(I), \Upsilon(J)]\right| \\
&=& \frac{q-1}{q^{m_2(p-1)}}+\frac{1}{p^{m_1(q-1)}q^{m_2(p-1)}} \\
&<& \frac{1}{10p^{m_1(q-1)}} \\
&<& \left|[i(J), \Upsilon(J)]\right|
\end{eqnarray*}
where we have used \eqref{20200805eq02} to arrive at the third line (see, Figure \ref{20200805Fig01}). This will imply that both $i(I) - i(J) >0$ and $d(J) - d(I) > 0$, so that $I \subseteq J$. 

\medskip

\begin{figure}[ht]
\begin{tikzpicture}[scale=8.5]
\draw (0,0) -- (1.2,0); 
\fill (0,0) circle [radius=.2pt];
\fill (1,0) circle [radius=.2pt];
\fill (0, -0.01) node [below] {$i(J)$};
\fill (1, -0.01) node [below] {$\Upsilon(J)$};
\fill (.9, 0) circle [radius=.2pt];
\fill (.9, -0.01) node [below] {$\textZeta(I)$};
\fill (1.1, 0) circle [radius=.2pt];
\fill (1.1, -0.01) node [below] {$d(I)$};
\fill (.2, 0) circle [radius=.2pt];
\fill (.2, -0.01) node [below] {$i(I)$};
\draw [decorate,decoration={brace,amplitude=10pt,raise=4pt},yshift=-.1pt] (.2,0) -- (1.1,0) node [black,midway,xshift=0cm, yshift=.8cm] { $I$};

\end{tikzpicture}
\caption{$I \subset J$}
\label{20200805Fig01}
\end{figure}
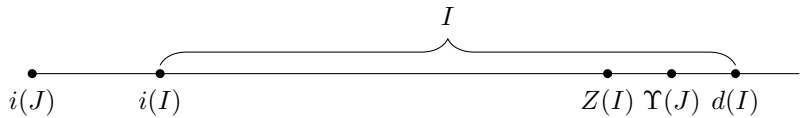

\medskip

\textit{Proof of (iii).} This is also clear since $I$ contains $\Upsilon(J)$ as an interior point (see, Figure \ref{20200805Fig01}) and all other $p$-adic intervals whose sidelength are less or equal to $|J|$ are either disjoint from $\Upsilon(J)$ or contain $\Upsilon(J)$ as an endpoint. 

\medskip

\textit{Proof of (iv).} The last assertion is straightforward from the choice of $m_1$. Indeed, 
$$
\Upsilon(J)-\textZeta(I)=\frac{1}{p^{m_1(q-1)}q^{m_2(p-1)}}<\varepsilon \cdot \frac{q}{q^{m_2(p-1)}}=\varepsilon |I|. 
$$
\end{proof}

\begin{rem}
The condition being used in \eqref{20200805eq01} is quite strong as it asserts that for such a chosen $k$ and $n_1$, we have infinitely many \emph{structured}  $j$ and $m_2$ (that is, $j$ belonging to an arithmetic progression) that give equality.  This implies that the corresponding far number inequality \eqref{defn01} for $\delta = k/p^{n_1}$ is sharp infinitely often in a structured way.  This indicates that for these $\delta$, not only does the sharp constant $C(\delta)$ gives precise geometric information about the proximity of the distinguished points in our construction, but that the sharp constant recurs infinitely often in a structured way.  This phenomenon might extend to other rational far numbers in a related way, and if not, might give a further gradient on which to determine the  \textquotedblleft strength" of a given far number. 
\end{rem}

\begin{proof} [Proof of Theorem \ref{20200722thm01}]
Let us choose an infinite collection of $\{\widetilde{J}^\ell\}_{\ell \ge 1}$ of pairwise disjoint $p$-adic subintervals of $[0, 1]$. For each $\widetilde{J}^\ell$, we associate a natural number $\alpha_\ell$ to it and apply Proposition \ref{20200805prop01} with $\varepsilon \left(=\varepsilon_\ell \right)=q^{-100\alpha_\ell}$ to $\widetilde{J}^\ell$, this yields a $q$-adic interval $I_\ell^{\alpha_\ell}$ and a $p$-adic interval $J^\ell$, which satisfy the third condition in Theorem \ref{20200722thm01}. While for the first condition, we note that for each $\ell \ge 1$, there holds
$$
I_\ell^{\alpha_\ell} \subseteq J^\ell \subset \widetilde{J}^\ell,
$$
and the disjointedness of $\{J^\ell\}_{\ell \ge 1}$ follows from the fact that $\{\widetilde{J}^\ell\}_{\ell \ge 1}$ are pairwise disjoint. Finally, for the second condition, it can be simply achieved by choosing $\alpha_\ell$'s in such a way. 
\end{proof}

\begin{rem}
Upon a very careful reading of this paper, one will hopefully discover that the location of the distinguished points do not matter; it is their relative order and proximity (specifically that they can be arbitrarily close) and consistency of position (that they are interior, an endpoint of a child, and in the same position each time) that matter.  For instance, $\Upsilon$ could be the right endpoint of the child $J_2$ and $\textZeta$ the right endpoint of child $I_9$ if $p=19$ and $q=13$.  However, the underlying number theory might prohibit such an arrangement of a uniform pattern; this is a key reason advocating for the approach that we chose.  In \cite{BMW}, the authors use a different approach, but it should be noted that they could have chosen any of the children's interior endpoints and orderings as well (there are only four possible total choices there), and either choice works due to a variant of their procedure akin to the approach described in our Remark \ref{going beyond far numbers}.   
\end{rem}

\section{The analysis part I: the construction of the measure $\mu$} \label{20200809sec01}
\label{analysis 1}
Theorem \ref{20200722thm01} plays the role of identifying each building block of the targeted measure. In this section, we construct a measure $\mu$ which is both $p$-adic and $q$-adic doubling, but not doubling. The proof of the fact that $\mu$ is $p$-adic will be postponed to the next section.  

From now on, we shall fix a $\ell \in \N$ and pay attention to a single $I_\ell^{\alpha_\ell}$ chosen in Theorem \ref{20200722thm01}, with an integer $\alpha_\ell \in \N$ being associated. The construction of the desired measure will be completed if we apply the construction in this section repeatedly to all $I_\ell^{\alpha_\ell}$'s and equip Lebesgue measure on the rest of $[0, 1) \backslash \left( \bigcup_{\ell} I_\ell^{\alpha_\ell} \right)$. 

Note that from the proof of Proposition \ref{20200805prop01}, we can indeed write
$$
I_\ell^{\alpha_\ell}=\left[\frac{j+1-q}{q^{m(p-1)}}, \frac{j+1}{q^{m(p-1)}} \right]
$$
for some $m \in \N$ and  
$$
j \in \left\{q-1, 2q-1, \dots, q^{m(p-1)}-1 \right\}.
$$
In particular, 
$$
\textZeta \left(I_\ell^{\alpha_\ell} \right)=\frac{j}{q^{m(p-1)}}.
$$
In the rest of this section, we write
$$
I:=I_\ell^{\alpha_\ell}, \ \alpha:=\alpha_\ell \quad \textrm{and} \quad  Z:=Z\left(I_\ell^{\alpha_\ell} \right)
$$
for convenience.

We remark that the number $\alpha$ comes into play a role in the construction. To begin with, we take $0<a<1<b$ such that
$$
(q-1)a+b=q.
$$
Roughly speaking, the idea is to assign the weights $a$ and $b$ carefully to the $q$-adic intervals near the point $\textZeta$. 

\begin{rem}
It is certainly possible that our analysis carries through by choosing $q$ weights $a_1, a_2, \dots , a_{q-1},b$; by doing things this way we have less constants to deal with.
\end{rem}

Here are some details. 

\medskip

\textit{Step $1$:} We start with the interval $I$ and all its $q$-adic children $\{I_1, \dots, I_q\}$ (see, \eqref{20200722eq02}). Define
$$
\mu(I_i)=a|I_i|=\frac{a|I|}{q}, \quad i=1, \dots, q-1.
$$
and
$$
\mu(I_q)=b|I_q|=\frac{b|I|}{q}. 
$$
Note that
\begin{equation} \label{20200810eq01}
\mu(I)=\sum_{i=1}^q \mu(I_i)=\frac{(q-1)a+b}{q}\cdot |I|=|I|. 
\end{equation} 
To this end, we denote
$$
H^{(1)}:=I_{q-1} \quad \textrm{and} \quad G^{(1)}:=I_q.
$$

\medskip

\textit{Step $2$:} We consider the $q$-adic children of $H^{(1)}$ and $G^{(1)}$, and denote them by
$$
\left\{H^{(1)}_1, \dots, H^{(1)}_q \right\}
$$
and 
$$
\left\{G^{(1)}_1, \dots, G^{(1)}_q \right\},
$$
respectively. We redistribute the weight on $H^{(1)}$ and $G^{(1)}$ by defining
$$
\mu\left(H_1^{(1)} \right)=\frac{b \mu \left(H^{(1)} \right)}{q}=\frac{ab |I|}{q^2} ,
$$
$$
\mu\left(H_i^{(1)} \right)=\frac{a \mu \left(H^{(1)} \right)}{q}=\frac{a^2 |I|}{q^2}, \quad i=2, \dots, q, 
$$
and
$$
\mu\left(G_1^{(1)} \right)=\frac{b \mu \left(G^{(1)} \right)}{q}=\frac{b^2 |I|}{q^2},
$$
$$
\mu\left(G_i^{(1)} \right)=\frac{a \mu \left(G^{(1)} \right)}{q}=\frac{ab |I|}{q^2}, \quad i=2, \dots, q.
$$
To this end, we denote 
$$
H^{(2)}:=H_q^{(1)} \quad \textrm{and} \quad G^{(2)}:=G_1^{(1)}. 
$$

\medskip

\textit{Step $k$, $3 \le k \le \alpha$:} Suppose we have already constructed $H^{(k-1)}$ and $G^{(k-1)}$, and our goal is to construct $H^{(k)}$ and $G^{(k)}$ and redistribute the weights on $H^{(k-1)}$ and $G^{(k-1)}$ from the previous step. Note that by induction we have
$$
\mu \left(H^{(k-1)} \right)=\frac{a^{k-1}|I|}{q^{k-1}} \quad \textrm{and} \quad \mu \left( G^{(k-1)} \right)=\frac{b^{k-1}|I|}{q^{k-1}}.
$$

Denote the $q$-adic children of $H^{(k-1)}$ and $G^{(k-1)}$ by
$$
\left\{H^{(k-1)}_1, \dots, H^{(k-1)}_q \right\}
$$
and 
$$
\left\{G^{(k-1)}_1, \dots, G^{(k-1)}_q \right\},
$$
respectively.  We redistribute the weight on $H^{(k-1)}$ and $G^{(k-1)}$ by defining
$$
\mu\left(H_1^{(k-1)} \right)=\frac{b \mu \left(H^{(k-1)} \right)}{q}=\frac{ba^{k-1} |I|}{q^k}
$$
$$
\mu\left(H_i^{(k-1)} \right)=\frac{a \mu \left(H^{(k-1)} \right)}{q}=\frac{a^k |I|}{q^k}, \quad i=2, \dots, q, 
$$
and
$$
\mu\left(G_1^{(k-1)} \right)=\frac{b \mu \left(G^{(k-1)} \right)}{q}=\frac{b^k |I|}{q^k},
$$
$$
\mu\left(G_i^{(k-1)} \right)=\frac{a \mu \left(G^{(k-1)} \right)}{q}=\frac{ab^{k-1} |I|}{q^k}, \quad i=2, \dots, q.
$$
To this end, we denote 
$$
H^{(k)}:=H_q^{(k-1)} \quad \textrm{and} \quad G^{(k)}:=G_1^{(k-1)}. 
$$

\medskip

\textit{Step $\alpha+1$:} From the above construction, we know that 
\begin{equation} \label{20200809eq01}
\mu\left(H^{(\alpha)} \right)=\frac{a^\alpha|I|}{q^\alpha} \quad \textrm{and} \quad \mu\left(G^{(\alpha)} \right)=\frac{b^\alpha|I|}{q^\alpha}.
\end{equation} 
We will now use a different way to distribute the weights $a$ and $b$ to the $q$-adic children of $H^{(\alpha)}$ and $G^{(\alpha)}$. Again, let
$$
\left\{H^{(\alpha)}_1, \dots, H^{(\alpha)}_q \right\}
$$
and 
$$
\left\{G^{(\alpha)}_1, \dots, G^{(\alpha)}_q \right\}
$$
be the $q$-adic children of $H^{(\alpha)}$ and $Q^{(\alpha)}$, respectively. We redistribute the weight on $H^{(\alpha)}$ and $G^{(\alpha)}$ by defining
$$
\mu\left(H_i^{(\alpha)} \right)=\frac{a \mu \left(H^{(\alpha)} \right)}{q}=\frac{a^{\alpha+1} |I|}{q^{\alpha+1}}, \quad i=1, \dots, q-1,
$$
$$
\mu\left(H_q^{(\alpha)} \right)=\frac{b \mu \left(H^{(\alpha)} \right)}{q}=\frac{a^\alpha b |I|}{q^{\alpha+1}},  
$$
and
$$
\mu\left(G_i^{(\alpha)} \right)=\frac{a \mu \left(G^{(\alpha)} \right)}{q}=\frac{ab^\alpha |I|}{q^{\alpha+1}}, \quad i=1, \dots, q-1,
$$
$$
\mu\left(G_q^{(\alpha)} \right)=\frac{b \mu \left(G^{(\alpha)} \right)}{q}=\frac{b^{\alpha+1} |I|}{q^{\alpha+1}}.
$$
To this end, we denote 
$$
H^{(\alpha+1)}:=H_q^{(\alpha)} \quad \textrm{and} \quad G^{(\alpha+1)}:=G_1^{(\alpha)}. 
$$

\medskip

\textit{Step $\alpha+k$, $2 \le k \le \alpha$}. Suppose we have already constructed $H^{(\alpha+k-1)}$ and $G^{(\alpha+k-1)}$ and similarly as before, our goal is to construction $H^{(\alpha+k)}$ and $G^{(\alpha+k)}$ and redistribute the weights on on $H^{(\alpha+k-1)}$ and $G^{(\alpha+k-1)}$ from the previous step. Note that by induction we have
$$
\mu \left(H^{(\alpha+k-1)} \right)=\frac{a^\alpha b^{k-1}|I|}{q^{\alpha+k-1}} \quad \textrm{and} \quad \mu \left( G^{(\alpha+k-1)} \right)=\frac{b^{\alpha} a^{k-1}|I|}{q^{\alpha+k-1}}.
$$

Denote the $q$-adic children of $H^{(\alpha+k-1)}$ and $G^{(\alpha+k-1)}$ by
$$
\left\{H^{(\alpha+k-1)}_1, \dots, H^{(\alpha+k-1)}_q \right\}
$$
and 
$$
\left\{G^{(\alpha+k-1)}_1, \dots, G^{(\alpha+k-1)}_q \right\},
$$
respectively.  We redistribute the weight on $H^{(\alpha+k-1)}$ and $G^{(\alpha+k-1)}$ by defining
$$
\mu\left(H_i^{(\alpha+k-1)} \right)=\frac{a \mu \left(H^{(\alpha+k-1)} \right)}{q}=\frac{a^{\alpha+1} b^{k-1}|I|}{q^{\alpha+k}}, \quad i=1, \dots, q-1,
$$
$$
\mu\left(H_q^{(\alpha+k-1)} \right)=\frac{b \mu \left(H^{(\alpha+k-1)} \right)}{q}=\frac{a^\alpha b^k|I|}{q^{\alpha+k}} 
$$
and
$$
\mu\left(G_i^{(\alpha+k-1)} \right)=\frac{a \mu \left(G^{(\alpha+k-1)} \right)}{q}=\frac{b^\alpha a^k |I|}{q^{\alpha+k}},  \quad i=1, \dots, q-1,
$$
$$
\mu\left(G_q^{(\alpha+k-1)} \right)=\frac{b \mu \left(G^{(\alpha+k-1)} \right)}{q}=\frac{b^{\alpha+1} a^{k-1} |I|}{q^{\alpha+k}}.
$$
To this end, we denote 
$$
H^{(\alpha+k)}:=H_q^{(\alpha+k-1)} \quad \textrm{and} \quad G^{(\alpha+k)}:=G_1^{(\alpha+k-1)}. 
$$

\medskip

The construction will stop at \emph{Step $2\alpha$}. 

\begin{rem}
The fact the construction goes to step $2\alpha$ (instead of, say, $\alpha$) is needed in order to show the measure is $p$-adic doubling, specifically in the case when $\textZeta \in J$ (precisely, one can see that the measure will fail to be $p$-adic doubling by stopping at Step $\alpha$).  Additionally, it allows us to exploit symmetry in the Case where $J$ is to the left of $\textZeta$, see Section \ref{20200815Sec01}.  Finally we remark that variations of our construction might be possible, as long as they stop at Step $c\alpha$, where $c$ is a constant independent of $\alpha$.
\end{rem}
\medskip

At the end of this section, we show that $\mu$ is not doubling but $q$-adic doubling, and we will prove that $\mu$ is $p$-adic doubling in next section. 

\begin{prop}
Let $\mu$ be defined as above. Then 
\begin{enumerate}
    \item [(1).] $\mu$ is not doubling;
    \item [(2).] $\mu$ is $q$-adic doubling. 
\end{enumerate}
\end{prop}

\begin{proof}
(1). By \eqref{20200809eq01}, we have  
$$
\frac{\mu\left(H^{(\alpha)}\right)}{\mu \left(G^{(\alpha)} \right)}=\left(\frac{a}{b} \right)^\alpha=\left(\frac{a}{b} \right)^{\alpha_\ell}. 
$$
Using Theorem \ref{20200722thm01}, (2), we see that this ratio can be arbitrary small when $\ell$ is sufficiently large, which will clearly fail Definition \ref{20200809defn01} if we consider the interval $H^{(\alpha)} \cup G^{(\alpha)}$. 

\medskip

(2). From the construction, it is clear that given any $q$-adic interval on $[0, 1)$, we have
$$
\frac{\mu(I_{i_1})}{\mu(I_{i_2})}=1, \frac{a}{b}, \ \textrm{or} \ \frac{b}{a}, \quad \textrm{for any} \ i_1, i_2 \in \left\{1, \dots, q \right\}, 
$$
in particular, this implies $\mu$ is $q$-adic.
\end{proof}

\section{The analysis part II: Trivial cases and exhaustion procedure}
\label{padic1}

The goal of the coming two sections is to prove the measure $\mu$ construction in Section \ref{20200809sec01} is $p$-adic doubling, and we start making some reduction in this section.

To begin with, let us take a $p$-adic interval $J \subset [0, 1)$, the goal is to show that there exists a positive constant $C>0$, such that
\begin{equation} \label{20200809eq02}
\frac{1}{C} \le \frac{\mu\left(J_{j_1} \right)}{\mu\left(J_{j_2} \right)} \le C, \quad \forall j_1, j_2 \in \{1, \dots, p\}.  
\end{equation} 

\subsection{Supporting constructions and trivial cases}

First, let us recall from Theorem \ref{20200722thm01} that, for each $I_\ell^{\alpha_\ell}$ which had been chosen, $J^\ell$ is the smallest $p$-adic interval that contains it, and the collection $\{J^\ell\}_{\ell \ge 1}$ is pairwise disjoint. We begin with three trivial cases.

\medskip

\textit{Trivial Case I:} $J$ does not intersect any $J^\ell$.

\medskip

\textit{Trivial Case II:} $J$ intersects more than two $J^\ell$'s. 

\medskip

\textit{Trivial Case III:} $J$ intersects a single $J^\ell$ but contains it strictly. 

\medskip

Indeed, one can see that in all of these cases, all the ratios in \eqref{20200809eq02} take the value $1$. More precisely, for the first case, since $J$ does not intersect any $J^\ell$, the measure $\mu$ restricted to $J$ is exactly the Lebesgue measure. While for the second case, since $J$ intersects more than two $J^\ell$'s, $J$ has to be a $p$-adic ancestor of those $J^\ell$'s which intersect with $J$. The desired claim for the second case follows from the fact that  $\mu(J^\ell)=|J^\ell|$ for all $\ell \ge 1$ (which follows from \eqref{20200810eq01} easily). Finally, the third case holds for the same reason. 

\medskip

Therefore, it suffices for us to consider the case \emph{when $J$ coincides with one of the $J^\ell$'s or with one of their $p$-adic offspring}. Again, let us fix some $\ell \in \N$, and write 
$$
 I:=I_\ell^{\alpha_\ell}, \alpha:=\alpha_\ell, Z:=Z\left(I_\ell^{\alpha_\ell}\right) \ \textrm{and} \ \Upsilon:=\Upsilon\left(J^\ell\right)
$$
for convenience. To begin with, let us observe the measure $\mu$ constructed in Section \ref{20200809sec01} in a more compact way. More precisely, we start with $H^{(2\alpha)}$ and $G^{(2\alpha)}$, and recall that
$$
\mu \left( H^{(2\alpha)} \right)=\mu \left( G^{(2\alpha)} \right)=\frac{a^\alpha b^\alpha |I|}{q^{2\alpha}}.
$$
We define the following: for any $1 \le k \le 2\alpha-1$
$$
F^{(k)}:=G^{(k)} \backslash G^{(k+1)} \quad \textrm{and} \quad E^{(k)}:=H^{(k)} \backslash H^{(k+1)}
$$
and 
$$
E^{(0)}:=I_1 \cup \dots \cup I_{q-2}.
$$
(see, Figure \eqref{20200824Fig01}).

\medskip

\begin{figure}[ht]
\begin{tikzpicture}[scale=5.5]
\draw (-.5,0) -- (1.55,0); 
\fill (.6, 0) circle [radius=.2pt];
\fill (.6, -.01) node [below] {\tiny $\textZeta$};
\fill (.66, 0) circle [radius=.2pt];
\fill (.66, -.01) node [below] {\tiny $\Upsilon$};
\fill (.8, 0) circle [radius=.2pt];
\draw [decorate,decoration={brace,amplitude=8pt,raise=0pt},yshift=.3pt] (.6, 0) -- (.8, 0) node [black,midway,xshift=0cm, yshift=.4cm] {\tiny $G^{(2\alpha)}$};
\fill (1.2, 0) circle [radius=.2pt];
\draw [decorate,decoration={brace,amplitude=8pt,raise=0pt},yshift=.3pt] (.8, 0) -- (1.2, 0) node [black,midway,xshift=0cm, yshift=.4cm] {\tiny $F^{(2\alpha-1)}$};
\fill (1.4, 0.01) node [above] {\small $\dots$};
\fill (.4, 0) circle [radius=.2pt];
\draw [decorate,decoration={brace,amplitude=8pt,raise=0pt},yshift=.3pt] (.4, 0) -- (.6 , 0) node [black,midway,xshift=0cm, yshift=.4cm] {\tiny $H^{(2\alpha)}$};
\fill (0, 0) circle [radius=.2pt];
\draw [decorate,decoration={brace,amplitude=8pt,raise=0pt},yshift=.3pt] (0, 0) -- (.4 , 0) node [black,midway,xshift=0cm, yshift=.4cm] {\tiny $E^{(2\alpha-1)}$};
\fill (-.25, .01) node [above]{\small $\dots$};
\end{tikzpicture}
\caption{$F^{(k)}$'s, $E^{(k)}$'s, $H^{(2\alpha)}$ and $G^{(2\alpha)}$.}
\label{20200824Fig01}
\end{figure}
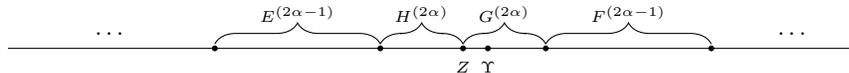

\begin{rem}
Let us make some remarks.
\begin{enumerate}
    \item [(1).] If $q=2$, then there is no need to consider $E^{(0)}$ since $E^{(0)}$ is  empty in this case;
    \item [(2).] One can also define $F^{(0)}$, but it is always empty.
\end{enumerate}
\end{rem}
By the construction in Section \ref{20200809sec01}, we have 
$$
\mu\left(F^{(k)} \right)=
\begin{cases}
(q-a) \cdot \frac{a^{k-\alpha} b^{\alpha}|I|}{q^{k+1}} , \quad \quad \quad \hfill \alpha \le k \le 2\alpha-1; \\
\medskip \\
(q-b) \cdot \frac{b^k |I|}{q^{k+1}}, \quad \quad \quad \hfill 1 \le k \le \alpha-1,
\end{cases}
$$
while 
$$
\left| F^{(k)} \right|=\frac{q-1}{q} \cdot \left|G^{(k)} \right|= \frac{(q-1)|I|}{q^{k+1}}, \quad 1 \le k \le 2\alpha-1.  
$$
Similarly, 
\begin{equation} \label{20200812eq01}
\mu\left(E^{(k)} \right)=
\begin{cases} 
(q-b) \cdot \frac{b^{k-\alpha} a^{\alpha}|I|}{q^{k+1}} , \quad \quad \quad \hfill \alpha \le k \le 2\alpha-1; \\
\medskip \\
(q-a) \cdot \frac{a^k |I|}{q^{k+1}}, \quad \quad \quad \hfill 1 \le k \le \alpha-1,
\end{cases}
\end{equation} 
while 
\begin{equation} \label{20200812eq02}
\left| E^{(k)} \right|=\frac{q-1}{q} \cdot \left|H^{(k)} \right|= \frac{(q-1)|I|}{q^{k+1}}, \quad 1 \le k \le 2\alpha-1.  
\end{equation} 
Finally, we have
$$
\mu\left(E^{(0)} \right)=\frac{a(q-2)|I|}{q} \quad \textrm{and} \quad 
\left| E^{(0)} \right|=\frac{(q-2)|I|}{q}. 
$$

Note that no $p$-adic interval can ever be equal to any of the $E^{(k)}$, $H^{(k)}$, $G^{(k)}$ or $F^{(k)}$.

We now consider one more easy case.

\subsubsection{When $J=J^\ell$.} \label{Case1} Let
$$
\{J_1, \dots, J_p\}
$$
be all the $p$-adic children of $J$, and note that in particular we have $r\left(J_1 \right)=l\left(J_2 \right)=\Upsilon$. We recall that the goal is to show \eqref{20200809eq02}. 

Since $I \subset J^\ell=J$ and $\Upsilon-\textZeta>0$, it follows that $I \subset J_1 \cup J_2$, and hence 
$$
\mu(J_i)=\frac{|J|}{p}, \quad j=3, \dots, p, 
$$
since we have $\{J^\ell\}_{\ell \ge 1}$ are pairwise disjoint. Moreover, by \eqref{20200805eq02}, we have
$$
\frac{\left|J_1 \right|}{2}=\frac{\left|J_2 \right|}{2}=\frac{|J|}{2p}>|I|.
$$
This implies the left half of $J_1$ and the right half of $J_2$ do not intersect with $I$, and combining this with the fact that $\mu(I)=|I|$, we have 
$$
\frac{|J|}{2p} \le \mu\left(J_i \right)\le \frac{2|J|}{p}, \quad i=1, 2, 
$$
and hence the ratio in \eqref{20200809eq02} is bounded above by $4$ and below by $\frac{1}{4}$, which implies the desired result.

\subsection{Visualization of the measure} Let us now turn to the non-trivial case of the proof, that is, the case \emph{the case when $J$ coincides with one of the offspring of a single $J^\ell$ and $J \subsetneq J^\ell$}. It will be convenient for us to see directly how the measure $\mu$ looks like. We start with visualizing $\mu$ on the left hand side of $Z$ (see, Figure \ref{20200811Fig01}).

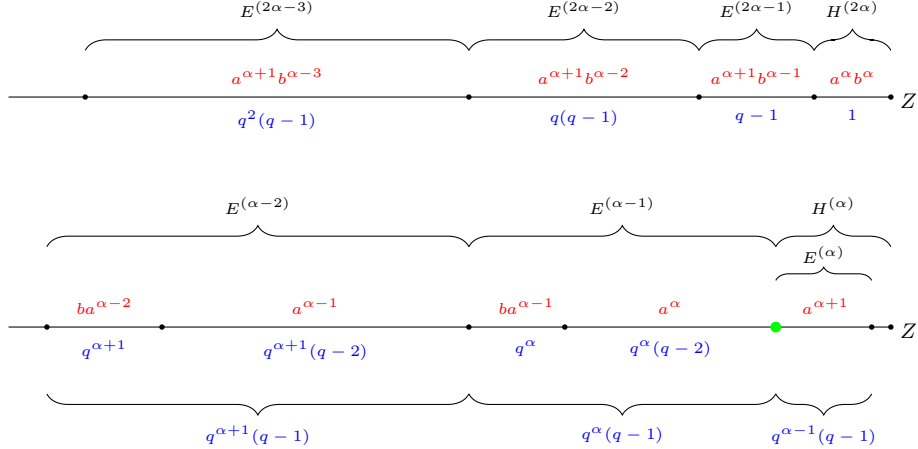
\begin{figure}[ht]
\begin{tikzpicture}[scale=5.1]
\draw (-.8,0) -- (1.5,0); 
\fill (1.5,0) circle [radius=.2pt];
\fill (1.5, -0.01) node [right] {{\footnotesize $\textZeta$}};
\fill (1.3,0) circle [radius=.2pt];
\fill (1,0) circle [radius=.2pt];
\fill (.4,0) circle [radius=.2pt];
\fill (-.6,0) circle [radius=.2pt];
\draw [decorate,decoration={brace,amplitude=8pt,raise=10pt},yshift=1.5pt] (1.3,0) -- (1.5,0) node [black,midway,xshift=0cm, yshift=.9cm] { {\tiny $H^{(2\alpha)}$}};
\draw [decorate,decoration={brace,amplitude=8pt,raise=10pt},yshift=1.5pt] (1.0,0) -- (1.3,0) node [black,midway,xshift=0cm, yshift=.9cm] { {\tiny $E^{(2\alpha-1)}$}};
\draw [decorate,decoration={brace,amplitude=8pt,raise=10pt},yshift=1.5pt] (.4,0) -- (1.0,0) node [black,midway,xshift=0cm, yshift=.9cm] { {\tiny $E^{(2\alpha-2)}$}};
\draw [decorate,decoration={brace,amplitude=8pt,raise=10pt},yshift=1.5pt] (-.6,0) -- (0.4,0) node [black,midway,xshift=0cm, yshift=.9cm] { {\tiny $E^{(2\alpha-3)}$}};
\fill (1.4, -.01) node [below] {{\color{blue}{\tiny $1$}}}; 
\fill (1.15, -.01) node [below] {{\color{blue}{\tiny $q-1$}}};
\fill (.7, -.01) node [below] {{{\color{blue}\tiny $q(q-1)$}}};
\fill (-.1, -.01) node [below] {{{\color{blue} \tiny $q^2(q-1)$}}};
\fill (1.4, .01) node [above] {{{\color{red} \tiny $a^\alpha b^\alpha$}}}; 
\fill (1.15, .01) node [above] {{{\color{red} \tiny $a^{\alpha+1} b^{\alpha-1}$}}};
\fill (.7, .01) node [above] {{{\color{red} \tiny $a^{\alpha+1} b^{\alpha-2}$}}};
\fill (-.1, .01) node [above] {{{\color{red}\tiny $a^{\alpha+1} b^{\alpha-3}$}}};
\draw (-.8,-.6) -- (1.5, -.6); 
\fill (1.5,-.6) circle [radius=.2pt];
\fill (1.5, -.61) node [right] {{\footnotesize $\textZeta$}};
\fill[green] (1.2,-.6) circle [radius=.4pt];
\fill (1.45, -.6) circle [radius=.2pt];
\draw [decorate,decoration={brace,amplitude=8pt,raise=10pt},yshift=4pt] (1.2,-.6) -- (1.5,-.6) node [black,midway,xshift=0cm, yshift=.9cm] { {\tiny $H^{(\alpha)}$}};
\draw [decorate,decoration={brace,amplitude=5pt,raise=10pt},yshift=1.5pt] (1.2,-.6) -- (1.45,-.6) node [black,midway,xshift=0cm, yshift=.7cm] { {\tiny $E^{(\alpha)}$}};
\fill (.4, -.6) circle [radius=.2pt]; 
\draw [decorate,decoration={brace,amplitude=8pt,raise=10pt},yshift=4pt] (.4,-.6) -- (1.2,-.6) node [black,midway,xshift=0cm, yshift=.9cm] { {\tiny $E^{(\alpha-1)}$}};
\fill (1.33, -.59) node [above] {{\color{red}\tiny $a^{\alpha+1}$}};
\fill (0.65, -.6) circle [radius=.2pt]; 
\fill (.55, -.61) node [below] {{\color{blue}\tiny $q^\alpha$}};
\fill (.925, -.61) node [below] {{\color{blue} \tiny $q^\alpha(q-2)$}}; 
\draw [decorate,decoration={brace,amplitude=8pt,raise=10pt},yshift=-3pt] (1.2,-.6) -- (.4,-.6) node [black,midway,xshift=0cm, yshift=-.9cm] {{{\color{blue}\tiny $q^\alpha(q-1)$}}};
\draw [decorate,decoration={brace,amplitude=8pt,raise=10pt},yshift=-3pt] (1.45,-.6) -- (1.2,-.6) node [black,midway,xshift=0cm, yshift=-.9cm] {{{\color{blue}\tiny $q^{\alpha-1}(q-1)$}}};
\fill (.55, -.59) node [above] {{\color{red}\tiny $ba^{\alpha-1}$}}; 
\fill (.925, -.59) node [above] {{\color{red} \tiny $a^\alpha$}}; 
\fill (-.7, -.6) circle [radius=.2pt]; 
\fill (-.4, -.6) circle [radius=.2pt]; 
\fill (-.55, -.61) node [below] {{\color{blue}\tiny $q^{\alpha+1}$}};
\fill (0, -.61) node [below] {{\color{blue}\tiny $q^{\alpha+1}(q-2)$}};
\draw [decorate,decoration={brace,amplitude=8pt,raise=10pt},yshift=-3pt] (.4,-.6) -- (-.7,-.6) node [black,midway,xshift=0cm, yshift=-.9cm] { {{\color{blue}\tiny $q^{\alpha+1}(q-1)$}}};
\draw [decorate,decoration={brace,amplitude=8pt,raise=10pt},yshift=4pt] (-.7,-.6) -- (.4,-.6) node [black,midway,xshift=0cm, yshift=.9cm] { {\tiny $E^{(\alpha-2)}$}};
\fill (-.55, -.59) node [above] {{\color{red}\tiny $ba^{\alpha-2}$}};
\fill (0, -.59) node [above] {{\color{red}\tiny $a^{\alpha-1}$}};
\end{tikzpicture}
\caption{$\mu$ on the left hand side of $\textZeta$.}
\label{20200811Fig01}
\end{figure}

Let us make some remarks for Figure \ref{20200811Fig01}. 

\begin{enumerate}
    \item [(1).] The red parts corresponds the \emph{weight} associated to each $E^{(k)}$, $1 \le k \le 2\alpha-1$ and $H^{(2\alpha)}$. For example, on $E^{(2\alpha-2)}$, we have the weight $a^{\alpha+1}b^{\alpha-2}$, which means $d \mu |_{E^{(2\alpha-2)}}=\left(a^{\alpha+1}b^{\alpha-2} \right) dx$, where $dx$ is the Lebesgue measure;
    
    \item [(2).] The blue parts refers to the \emph{ratio} of the lengths between the targeted interval and $H^{(2\alpha)}$ (see, \eqref{20200812eq02}). For example, under $E^{(2\alpha-2)}$, we have the ratio $q(q-1)$, which means $$\frac{\left|E^{(2\alpha-2)} \right|}{\left|H^{(2\alpha)}\right|}=q(q-1);$$ 
    
    \item [(3).] The behavior of the measure $\mu$ follows two different patterns on the left hand side of $\textZeta$, with the distinguished point $l\left(H^{\alpha} \right)$ (the green point in Figure \ref{20200811Fig01}) and this corresponds the fact that we distribute the weight in a different way from \emph{Step $\alpha+1$} onward (see, Section \ref{20200809sec01}). 
    
    More precisely, when $1 \le k \le \alpha-1$, the weight and ratio associated to $E^{(k)}$ is given by Figure \ref{20200811Fig02}:

    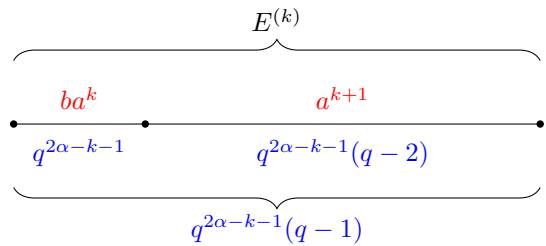
\begin{figure}[ht]
\begin{tikzpicture}[scale=7]
\draw (0, 0)--(1, 0);
\fill (0, 0) circle [radius=.2pt];
\fill (1, 0) circle [radius=.2pt];
\fill (.25, 0) circle [radius=.2pt];
\fill (.125, -.01) node [below] {{\color{blue} $q^{2\alpha-k-1}$}};
\fill (.625, -.01) node [below] {{\color{blue} $q^{2\alpha-k-1}(q-2)$}};
\draw [decorate,decoration={brace,amplitude=8pt,raise=10pt},yshift=-2pt] (1, 0) -- (0, 0) node [black,midway,xshift=0cm, yshift=-.9cm] { {{\color{blue} $q^{2\alpha-k-1}(q-1)$}}};
\fill (.125, .01) node [above] {{\color{red} $ba^k$}};
\fill (.625, .01) node [above] {{\color{red} $a^{k+1}$}};
\draw [decorate,decoration={brace,amplitude=8pt,raise=10pt},yshift=2pt] (0, 0) -- (1, 0) node [black,midway,xshift=0cm, yshift=.9cm]  {$E^{(k)}$};
\end{tikzpicture}
\caption{$E^{(k)}$ with $1 \le k \le \alpha-1$.}
\label{20200811Fig02}
\end{figure}

 and when $\alpha \le k \le 2\alpha-1$ it is given by Figure \ref{20200811Fig03}:

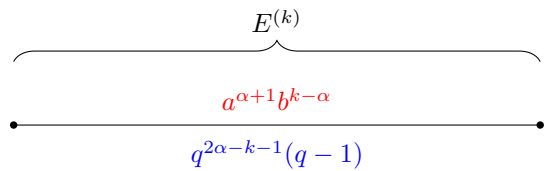
\begin{figure}[ht]
\begin{tikzpicture}[scale=7]
\draw (0, 0)--(1, 0);
\fill (0, 0) circle [radius=.2pt];
\fill (1, 0) circle [radius=.2pt];
\fill (.5, -.01) node [below] {{\color{blue} $q^{2\alpha-k-1}(q-1)$}};
\fill (.5, .01) node [above] {{\color{red} $a^{\alpha+1}b^{k-\alpha}$}};
\draw [decorate,decoration={brace,amplitude=8pt,raise=10pt},yshift=2pt] (0, 0) -- (1, 0) node [black,midway,xshift=0cm, yshift=.9cm]  {$E^{(k)}$};
\end{tikzpicture}
\caption{$E^{(k)}$ with $\alpha \le k \le 2\alpha-1$.}
\label{20200811Fig03}
\end{figure}

    \item[(4).] We can easily recover the previous calculation \eqref{20200812eq02} by using these figures. Indeed, when $1 \le k \le \alpha-1$,
    \begin{eqnarray*}
    \mu\left(E^{(k)} \right)%
    &=& ba^k \cdot q^{2\alpha-k-1} \cdot \frac{|I|}{q^{2\alpha}}+a^{k+1} \cdot q^{2\alpha-k-1}(q-2) \cdot \frac{|I|}{q^{2\alpha}} \\
    &=& \left(b+a(q-2) \right) \cdot \frac{a^k |I|}{q^{k+1}} \\
    &=&  \left(q-a \right) \cdot \frac{a^k |I|}{q^{k+1}}.
    \end{eqnarray*}
  and when $\alpha \le k \le 2\alpha-1$, we have
  \begin{eqnarray*}
  \mu\left(E^{(k)} \right)%
  &=& a^{\alpha+1} b^{k-\alpha} \cdot q^{2\alpha-k-1}(q-1) \cdot \frac{|I|}{q^{2\alpha}} \\
  &=& (q-1)a \cdot \frac{b^{k-\alpha} a^{\alpha}|I|}{q^{k+1}} \\
  &=& (q-b) \cdot \frac{b^{k-\alpha} a^{\alpha}|I|}{q^{k+1}}. 
  \end{eqnarray*}
   
\end{enumerate}

Similarly, we can also plot $\mu$ on the right hand side of $Z$ as follows (see, Figure \ref{20200815Fig01}).

\begin{figure}[ht]
\begin{tikzpicture}[scale=5.5]
\draw (-.8,0) -- (1.6,0); 
\fill (-.8,0) circle [radius=.2pt];
\fill (-.8, -0.01) node [left] {{\footnotesize $\textZeta$}};
\fill [darkorchid] (-.75, 0) circle [radius=.2pt];
\fill (-.75, -0.01) node [below] {\footnotesize $\Upsilon$}; 
\fill (-.6, 0) circle [radius=.2pt]; 
\fill (-.1, 0) circle [radius=.2pt];
\fill (.5, 0) circle [radius=.2pt];
\fill (1.5, 0) circle [radius=.2pt]; 
\draw [decorate,decoration={brace,amplitude=8pt,raise=10pt},yshift=1.5pt] (-.8,0) -- (-.6,0) node [black,midway,xshift=0cm, yshift=.9cm] { {\tiny $G^{(2\alpha)}$}};
\fill (-.7, .01) node [above] {\tiny {\color{red}$b^\alpha a^\alpha$}};
\draw [decorate,decoration={brace,amplitude=8pt,raise=10pt},yshift=-3pt] (-.6, 0) -- (-.8, 0) node [black,midway,xshift=0cm, yshift=-.9cm] {{{\color{blue}\tiny $1$}}};
\fill (-.3, 0) circle [radius=.2pt]; 
\draw [decorate,decoration={brace,amplitude=8pt,raise=10pt},yshift=1.5pt] (-.6,0) -- (-.1,0) node [black,midway,xshift=0cm, yshift=.9cm] { {\tiny $F^{(2\alpha-1)}$}};
\fill (-.45, .01) node [above] {\tiny {\color{red} $b^\alpha a^\alpha$}}; 
\fill (-.2, .01) node [above] {\tiny {\color{red} $b^{\alpha+1}{a^{\alpha-1}}$}}; 
\draw [decorate,decoration={brace,amplitude=8pt,raise=10pt},yshift=-3pt] (-.1, 0) -- (-.6, 0) node [black,midway,xshift=0cm, yshift=-.9cm] {{{\color{blue}\tiny $q-1$}}};
\fill (-.45, -.01) node [below] {\tiny {\color{blue} $q-2$}}; 
\fill (-.2, -.01) node [below] {\tiny {\color{blue} $1$}}; 
\fill (.3, 0) circle [radius=.2pt];
\draw [decorate,decoration={brace,amplitude=8pt,raise=10pt},yshift=1.5pt] (-.1,0) -- (.5,0) node [black,midway,xshift=0cm, yshift=.9cm] { {\tiny $F^{(2\alpha-2)}$}};
\fill (.1, .01) node [above] {\tiny {\color{red} $b^{\alpha} a^{\alpha-1}$}}; 
\fill (.1, -0.01) node [below] {\tiny {\color{blue} $q(q-2)$}};
\fill (.4, -0.01) node [below] {\tiny {\color{blue} $q$}};
\fill (.4, .01) node [above] {\tiny {\color{red} $b^{\alpha+1}a^{\alpha-2}$}}; 
\draw [decorate,decoration={brace,amplitude=8pt,raise=10pt},yshift=-3pt] (.5, 0) -- (-.1, 0) node [black,midway,xshift=0cm, yshift=-.9cm] {{{\color{blue}\tiny $q(q-1)$}}};
\fill (1.3, 0) circle [radius=.2pt]; 
\fill (.9, -.01) node [below] {\tiny {\color{blue} $q^2(q-2)$}};
\fill (1.4, -.01) node [below] {\tiny {\color{blue} $q^2$}};
\draw [decorate,decoration={brace,amplitude=8pt,raise=10pt},yshift=-3pt] (1.5, 0) -- (.5, 0) node [black,midway,xshift=0cm, yshift=-.9cm] {{{\color{blue}\tiny $q^2(q-1)$}}};
\fill (.9, .01) node [above] {\tiny {\color{red} $b^\alpha a^{\alpha-2}$}};
\fill (1.4, .01) node [above] {\tiny {\color{red} $b^{\alpha+1} a^{\alpha-3}$}}; 
\draw [decorate,decoration={brace,amplitude=8pt,raise=10pt},yshift=1.5pt] (.5,0) -- (1.5,0) node [black,midway,xshift=0cm, yshift=.9cm] { {\tiny $F^{(2\alpha-3)}$}};
\draw (-.8,-.8) -- (1.6, -.8); 
\fill (-.8, -.8) circle [radius=.2pt];
\fill (-.8, -.81) node [left] {{\footnotesize $\textZeta$}};
\fill [darkorchid] (-.75, -.8) circle [radius=.2pt];
\fill (-.75, -0.81) node [below] {\footnotesize $\Upsilon$}; 
\fill [green] (-.2, -.8) circle [radius=.4pt]; 
\fill (-.7, -.8) circle [radius=.2pt];
\fill (-.4, -.8) circle [radius=.2pt];
\fill (-.3, -.79) node [above] {\tiny {\color{red} $b^{\alpha+1}$}};
\fill (-.55, -.79) node [above] {\tiny {\color{red} $b^\alpha a$}}; 
\draw [decorate,decoration={brace,amplitude=8pt,raise=10pt},yshift=5pt] (-.8,-.8) -- (-.2,-.8) node [black,midway,xshift=0cm, yshift=.9cm] { {\tiny $G^{(\alpha)}$}};
\draw [decorate,decoration={brace,amplitude=8pt,raise=10pt},yshift=1.3pt] (-.7,-.8) -- (-.2,-.8) node [black,midway,xshift=0cm, yshift=.9cm] { {\tiny $F^{(\alpha)}$}};
\fill (-.3, -.81) node [below] {\tiny {\color{blue} $q^{\alpha-1}$}};
\fill (-.55, -.81) node [below] {\tiny {\color{blue} $q^{\alpha-1}(q-2)$}};
\draw [decorate,decoration={brace,amplitude=8pt,raise=10pt},yshift=-3pt] (-.2, -.8) -- (-.7, -.8) node [black,midway,xshift=0cm, yshift=-.9cm] {{{\color{blue}\tiny $q^{\alpha-1}(q-1)$}}};
\fill (.6, -.8) circle [radius=.2pt]; 
\fill (1.5, -.8) circle [radius=.2pt]; 
\draw [decorate,decoration={brace,amplitude=8pt,raise=10pt},yshift=5pt] (-.2,-.8) -- (.6,-.8) node [black,midway,xshift=0cm, yshift=.9cm] { {\tiny $F^{(\alpha-1)}$}};
\draw [decorate,decoration={brace,amplitude=8pt,raise=10pt},yshift=5pt] (.6,-.8) -- (1.5,-.8) node [black,midway,xshift=0cm, yshift=.9cm] { {\tiny $F^{(\alpha-2)}$}};
\fill (.2, -.79) node [above] {\tiny {\color{red} $b^{\alpha-1}a$}};
\draw [decorate,decoration={brace,amplitude=8pt,raise=10pt},yshift=-3pt] (.6, -.8) -- (-.2, -.8) node [black,midway,xshift=0cm, yshift=-.9cm] {{{\color{blue}\tiny $q^{\alpha}(q-1)$}}};
\draw [decorate,decoration={brace,amplitude=8pt,raise=10pt},yshift=-3pt] (1.5, -.8) -- (.6, -.8) node [black,midway,xshift=0cm, yshift=-.9cm] {{{\color{blue}\tiny $q^{\alpha+1}(q-1)$}}};
\fill (1.05, -.79) node [above] {\tiny {\color{red} $b^{\alpha-2}a$}};
\end{tikzpicture}
\caption{$\mu$ on the right side of $\textZeta$.}
\label{20200815Fig01}
\end{figure}
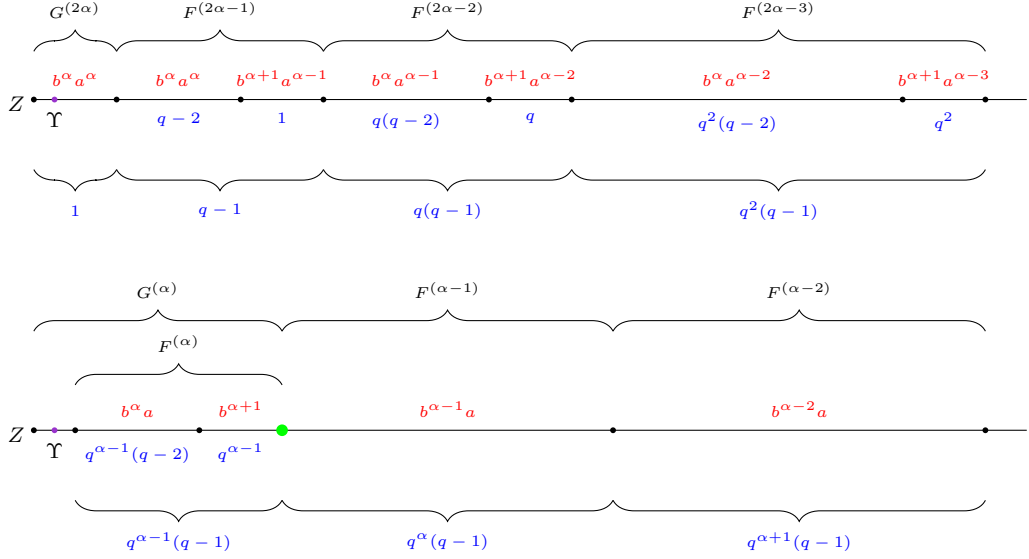

\bigskip

The building block of $\mu$ on the right side of $Z$ can be visualized as follows. 

\medskip

\begin{enumerate}
    \item [(1).] When $1 \le k \le \alpha-1$, the weights and the ratio associated to $F^{(k)}$ is given by (see, Figure \ref{20200815Fig02})

        \begin{figure}[ht]
\begin{tikzpicture}[scale=7]
\draw (0, 0)--(1, 0);
\fill (0, 0) circle [radius=.2pt];
\fill (1, 0) circle [radius=.2pt];
\fill (.5, .01) node [above] {{\color{red} $b^k a$}};
\fill (.5, -.01) node [below] {{\color{blue} $q^{2\alpha-k-1}(q-1)$}}; 
\draw [decorate,decoration={brace,amplitude=8pt,raise=10pt},yshift=2pt] (0, 0) -- (1, 0) node [black,midway,xshift=0cm, yshift=.9cm]  {$F^{(k)}$};
\end{tikzpicture}
\caption{$F^{(k)}$ with $1 \le k \le \alpha-1$.}
\label{20200815Fig02}
\end{figure}
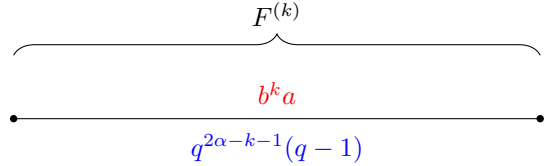

\bigskip

\item [(2).] When $\alpha \le k \le 2\alpha-1$, we have the following (see, Figure \ref{20200815Fig03})
    
    \begin{figure}[ht]
\begin{tikzpicture}[scale=7]
\draw (0, 0)--(1, 0);
\fill (0, 0) circle [radius=.2pt];
\fill (1, 0) circle [radius=.2pt];
\fill (.75, 0) circle [radius=.2pt];
\fill (.375, -.01) node [below] {{\color{blue} $q^{2\alpha-k-1}(q-2)$}};
\fill (.875, -.01) node [below] {{\color{blue} $q^{2\alpha-k-1}$}};
\draw [decorate,decoration={brace,amplitude=8pt,raise=10pt},yshift=-2pt] (1, 0) -- (0, 0) node [black,midway,xshift=0cm, yshift=-.9cm] { {{\color{blue} $q^{2\alpha-k-1}(q-1)$}}};
\fill (.375, .01) node [above] {{\color{red} $b^\alpha a^{k-\alpha+1}$}};
\fill (.875, .01) node [above] {{\color{red} $b^{\alpha+1}a^{k-\alpha}$}};
\draw [decorate,decoration={brace,amplitude=8pt,raise=10pt},yshift=2pt] (0, 0) -- (1, 0) node [black,midway,xshift=0cm, yshift=.9cm]  {$F^{(k)}$};
\end{tikzpicture}
\caption{$F^{(k)}$ with $\alpha \le k \le 2\alpha-1$.}
\label{20200815Fig03}
\end{figure}
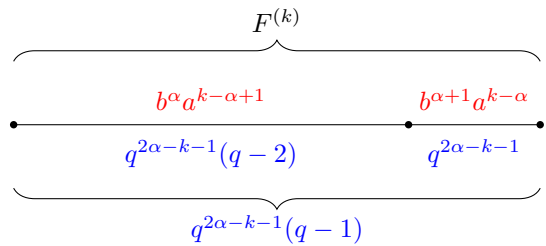
\end{enumerate}

\begin{rem}
The above figures clearly indicate that this case is a ``mirror symmetric" version of the case where $\textZeta$ is to the right of $J_p$.
\end{rem}
\subsection{The exhaustion procedure}
\label{exhaustion procedure}
Here we describe the main idea that we will apply to all the rest of the scenarios, called the \emph{exhaustion procedure}.  Upon reading this section, the reader should be equipped to check the calculations with greater ease, and also be convinced that this technique will handle all the nontrivial cases, and therefore show that our measure is $p$-adic doubling.  At its heart, this idea capitalizes on the geometric progressions inherent in this construction to bound the ratios of $J's$ $p$-adic children.

The basic idea is the following: assuming $Z$ is on the right hand side of $J$,  we will look at the rightmost child of $J$, $J_p$, which will intersect a certain number of the $E$ and $H^{(2\alpha)}$ intervals (or none of them) which we have defined.  These intervals have a nice structure, in particular they exhibit geometric grown (see, Figure \ref{20200811Fig01}), and so $\mu(J_p)$ will be (essentially) controlled by 
\begin{equation} \label{20200824eq01}
\left(\textrm{the weight on the leftmost part of} \ J_p \right) \cdot \left|J_p \right|.
\end{equation}
All other children are limited in how many other $E$ and $H^{(2\alpha)}$ intervals they can intersect, and we can ``exhaust" these children one by one by quantifying how many can lie in the next largest interval.  Once we exhaust a certain number of children, we move again to the next largest interval, and exhaust more.  Due to the geometric progression inherent in certain ratios involving either the $E$ or $H^{(2\alpha)}$ intervals, we will be able to exhaust all the children in $N$ steps, where $N$ is a fixed number depending only on $p$ and $q$ (see, \eqref{20200813eq11}). Moreover, the geometric progression also guarantees that each $J_i, 1 \le i \le p-1$ intersects at most two of the $E$ and $H$ intervals.  

As long as we can favorably compare ratios at each exhaustion, we will have upper and lower bounds that are controlled by a fixed power (no matter what $\alpha$ is) of the ratios at a single exhaustion.  Taking care of the special case in considering the ratio between $\mu(J_p)$ and $\mu(J_{p-1})$, which is calculated separately, all other children will be exhausted, resulting in the whole family being exhausted, and leading to the explicitly calculable (though by no means optimal) upper and lower bounds. 

Some details are in order.  Here, we treat a generic situation to give greater unity, there are a few small technicalities and particularities which are pointed out later when they occur.  First we notice that $N:=\floor*{\frac{\log p}{\log q}}+1$ is the smallest integer, such that
\begin{equation} \label{20200813eq11}
p<1+(q-1)+q(q-1)\dots+q^{N-1}(q-1)=q^N.
\end{equation} 
We will end up doing at most $N$ exhaustion steps, where at the $k$th exhaustion, we will at least exhaust $q^{k-1}(q-1)$ children, precisely those consecutively located to the right of the previous exhaustion.  So explicitly, for the first exhaustion, at least $q$ children are exhausted, for the second, at least $q(q-1)$, for the third, at least $q^2(q-1)$ are, and so on.  Since the total number of children is $p$, which is less than $q^N$ (see \eqref{20200813eq11}), and remember that we can count $J_p$ as exhausted already as we will handle it separately), after the $N$-th exhaustion, all the $p$-adic children will be exhausted.

At each exhaustion step, the ratio of the (to be) exhausted children to each other will be controlled by some bounded constant dependent on $a$ and $b$ (or on $q$ and $p$, but never $\alpha$), like $b/a^2$, so after $N$ steps, the worst ratio we can have between all children is that constant to the $N$th power, such as $(b/a^2)^N$.  This essentially leaves only one remaining step: to calculate the ratio 
$$
\frac{\mu(J_p)}{\mu(J_{p-1})}.
$$ 
Computationally, this full procedure takes $N+1$ steps, and more importantly, the number of steps as well as the actual ratios computed, are independent of the parameter $\alpha$.

\begin{rem}
Our approach involving the exhaustion procedure is different from \cite{BMW}, though some similar ideas are present in both works.  We developed this exhaustion procedure due to the generality that we consider, but it also allows us to unify several of the cases under one umbrella, and this guides the different way that we split up our cases.  In particular, we do not need the concept of \emph{valuable set} from \cite{BMW}.  
\end{rem}

We now work the details out. Recall that $J$ is a $p$-adic interval which coincides with one of the $p$-adic offspring of $J^\ell$ and $J \subsetneq J^\ell$. Moreover, we denote all the $p$-adic children of $J$ by
$$
\left\{J_1, \dots, J_p\right\},
$$
where each $J_i, 1 \le i \le p$ is defined as in \eqref{20200722eq02}.

\medskip

\textbf{\textit{Assumptions:}} To begin with, we make some assumptions.  

\begin{enumerate}
    \item [(1).] We assume that $Z$ is on the right hand side of $J$. Note that in this case, among all $\{J_1, \dots, J_p\}$, $J_p$ is the interval which is the closest to $Z$, and intersects the $E^{(i)}$'s and $H^{(2\alpha)}$ in the most complicated way. 
    
    \medskip
    
    \item [(2).] There exists some constant $A>0$ which is independent of $\alpha$, such that
    \begin{equation} \label{20200825eq01}
    \frac{1}{A} \le \frac{\mu(J_p)}{\mu(J_{p-1})} \le A.
    \end{equation} 
    
    \medskip
    
    \item [(3).] For simplicity, let us write
     $$
    E^{(2\alpha)}:=H^{(2\alpha)}.  
     $$ 
    
    \medskip
    
    \item [(4).] Finally, we assume 
     \begin{equation} \label{20200825eq02}
    l \left(J_p \right) \in E^{(K)}, 
    \end{equation}
    for some $K \in \{0, \dots, 2 \alpha\}$.
     \end{enumerate}

\medskip

\begin{rem}
    
    We make a remark the condition \eqref{20200825eq01}, together with condition \eqref{20200825eq02}, can be interpreted as a quantitative way to capture all the information coming from $J_p$, and from now on, it suffices for us to deal with $\{J_1, \dots J_{p-1}\}$. 
    
    Moreover, it can happen that the condition \eqref{20200825eq02} fails, that is, $l \left(J_p \right) \notin I$, and note that the exhaustion procedure in this case is indeed trivial (see, Section \ref{20200831case01}). 
\end{rem}

\begin{rem}
Let us also make some remarks for the other cases.

\begin{enumerate}
\item [(a).] If $Z \in J$, we will see this case is essentially the ``same" as the case when $Z$ is on the right hand side of $J$ since by our construction
$$
\Upsilon-\textZeta<q^{-100\alpha} |I| \ll q^{-2\alpha} |I|=\left|G^{(2\alpha)}  \right|.
$$

\medskip

\item [(b).] If $Z$ is on the left hand side of $J$, then the assumption \eqref{20200825eq01} above would be:  there exists some constant $A>0$, independent of $\alpha$, such that 
    \begin{equation} \tag{\ref{20200825eq01}$^\prime$} 
      \frac{1}{A} \le \frac{\mu(J_1)}{\mu(J_2)} \le A.
    \end{equation} 
and \eqref{20200825eq02} would be 
 \begin{equation} \tag{\ref{20200825eq02}$^\prime$} 
    F^{(2\alpha)}:=G^{(2\alpha)}. 
    \end{equation} 
Hence, the exhaustion procedure in this case be treated a ``mirror symmetric" version of the case we are considering, with respect to $Z$, in which, we shall start with $J_1$, instead of $J_p$ as in this case $J_1$ is the closest interval to $\textZeta$. 
\end{enumerate}
\end{rem}

\medskip

\textbf{\textit{Step 1:}} By \eqref{20200825eq02}, $J_p \subset H^{(k)}$. Since 
$$
\frac{\left|E^{(K-1)}\right|}{\left|H^{(K)}\right|}=q-1, 
$$
it implies that the $q-1$ $p$-adic children to the left of $J_p$, that is,
$$
J_{p-1}, \dots, J_{p-(q-1)}
$$
are either contained in $E^{(K)}$, $E^{(K-1)}$ or $E^{(K)} \cup E^{(K-1)}$ (at most one of them). Therefore, the values of 
$$
\mu(J_{p-1}), \dots, \mu(J_{p-(q-1)})
$$
are either 
$$
\left(\textrm{weights on} \ E^{(K)} \right) \cdot |J_p|, \quad \left(\textrm{weights on} \ E^{(K-1)} \right) \cdot |J_p| 
$$
or a convex combination of them. Moreover, it is not hard to see that 
$$
\frac{a^2}{b} \le \frac{\mu\left(J_{j_1}\right)}{\mu\left(J_{j_2}\right)} \le \frac{b}{a^2}, \quad j_1, j_2 \in \{p-(q-1), \dots, p\}
$$
(one may refer Figure \ref{20200811Fig01} to check this). Finally, we note that here we exhausted the rightmost $q = 1+(q-1)$ many $p$-adic children of $J$, and this corresponds to the term``$1+(q-1)$" in \eqref{20200813eq11}. 

\medskip

\textbf{\textit{Step 2:}} Now we move the next step to exhaust more $p$-adic children of $J$ from the rightmost side. Note that
Since
$$
\frac{\left|E^{(K-2)}\right|}{\left|H^{(K)}\right|}=q(q-1), \quad \textrm{and} \quad
\frac{\left|E^{(K-2)}\right|}{\left|H^{(K-1)}\right|}=q-1, 
$$
we conclude that the next $q(q-1)$ $p$-adic children next to $J_{p-q+1}$, that is,
$$
J_{p-q}, \dots, J_{p-q^2+1}
$$
will be of one of the following situations:
\begin{enumerate}
    \item [(1).] They are either contained in $E^{(K)}$, $E^{(K-1)}$ or $E^{(K)} \cup E^{(K-1)}$ (at most one of them);
    
    \medskip
    
    \item [(2).] They are either contained in $E^{(K-1)}$, $E^{(K-2)}$ or $E^{(K-1)} \cup E^{(K-2)}$ (at most one of them).
\end{enumerate}
This implies the values of 
$$
\mu \left(J_{p-q} \right), \dots \mu \left(J_{p-q^2+1} \right)
$$
will take one of the following forms: 
\begin{enumerate}
\item [(1').] Either
$$
\left(\textrm{weights on} \ E^{(K)} \right) \cdot |J_p|, \quad \left(\textrm{weights on} \ E^{(K-1)} \right) \cdot |J_p|,
$$
or a convex combination of them;

\medskip

\item [(2').] Either 
$$
\left(\textrm{weights on} \ E^{(K-1)} \right) \cdot |J_p|, \quad \left(\textrm{weights on} \ E^{(K-2)} \right) \cdot |J_p|,
$$
or a convex combination of them.
\end{enumerate}

Moreover, it still holds that 
$$
\frac{a^2}{b} \le \frac{\mu\left(J_{j_1}\right)}{\mu\left(J_{j_2}\right)} \le \frac{b}{a^2}, \quad j_1, j_2 \in \{p-q^2+1, \dots, p-q\}
$$
To this end, we note that in this step we exhausted the rightmost $q(q-1)$ many $p$-adic children of $J$ next to $J_{p-q+1}$, and this corresponds to the term ``$q(q-1)$" in \eqref{20200813eq11}. 

\medskip

\textbf{\textit{Step k:}} In general, assume we have already made such an exhaustion $k-1$ times, and here is how we make the $k$-th exhaustion. Observe that
$$
\frac{\left|E^{(K-k)}\right|}{\left|H^{(K)}\right|}=q^{k-1}(q-1), \quad  \frac{\left|E^{(K-k)}\right|}{\left|H^{(K-1)}\right|}=q^{k-2}(q-1), \quad  \dots, \quad  \frac{\left|E^{(K-k)}\right|}{\left|H^{(K-k+1)}\right|}=q-1.
$$
Similarly, these allows us to conclude that 
$$
\mu\left(J_{p-q^{k-1}}\right), \dots, \mu\left(J_{p-q^k+1} \right)
$$
will be of one of the following situations:
\begin{enumerate}
    \item [(1).]  Either
$$
\left(\textrm{weights on} \ E^{(K_1)} \right) \cdot |J_p|, \quad \left(\textrm{weights on} \ E^{(K_1-1)} \right) \cdot |J_p|,
$$
or a convex combination of them;

\medskip

$\vdots$

\medskip

\item [(k).] Either
$$
\left(\textrm{weights on} \ E^{(K_1-k+1)} \right) \cdot |J_p|, \quad \left(\textrm{weights on} \ E^{(K_1-k)} \right) \cdot |J_p|,
$$
or a convex combination of them.
\end{enumerate}
Most importantly, the estimate 
$$
\frac{a^2}{b} \le \frac{\mu\left(J_{j_1}\right)}{\mu\left(J_{j_2}\right)} \le \frac{b}{a^2}, \quad j_1, j_2 \in \{p-q^k+1, \dots, p-q^{k-1}\}
$$
still holds, and just as before, we have exhausted the rightmost $q^{k-1}(q-1)$ many $p$-adic children of $J$ next to $J_{p-q^{k-1}+1}$, and this corresponds to the term ``$q^{k-1}(q-1)$" in \eqref{20200813eq11}. 

\medskip

\textbf{\textit{Step N:}} Continuing this process and \eqref{20200813eq11} suggests that this process will stop after $N$ steps, that is, all the $p$-adic children of $J$ will be exhausted after $N$ steps. Recall that $N$ only depends on $p$ and $q$. This suggests  \eqref{20200809eq02} holds with the absolute constant $C=\left(\frac{b}{a^2} \right)^{N}$, that is, 
\begin{equation} \label{20200814eq01}
\left(\frac{a^2}{b} \right)^{N} \le \frac{\mu(J_{j_1})}{\mu(J_{j_2})} \le \left(\frac{b}{a^2} \right)^{N}, \quad j_1, j_2 \in \{1, \dots, p-1\}.
\end{equation} 

\medskip

\textbf{\textit{Final Step:}} The final step would be adding $J_p$ to \eqref{20200814eq01}, via our assumption \eqref{20200825eq01}. This is straightforward by both estimates, and finally we conclude that
$$
\frac{1}{A} \cdot \left(\frac{a^2}{b} \right)^{N} \le  \frac{\mu(J_{j_1})}{\mu(J_{j_2})} \le A \cdot \left(\frac{b}{a^2} \right)^{N}, \quad j_1, j_2 \in \{1, \dots, p\}.
$$

\bigskip

Therefore, we see that the exhaustion procedure reduces the original problem to the computation of the ratio $\frac{\mu(J_{j_p})}{\mu(J_{j_{p-1}})}$.

\medskip
\section{The analysis part III: Computation of $\frac{\mu(J_p)}{\mu(J_{p-1})}$}
\label{padic2}
In this section, we complete the proof of $\mu$ is $p$-adic doubling by showing that the constant $A$ in our assumption \eqref{20200825eq01} can be chosen only depending on $p$, $q$, $a$ and $b$. Without loss of generality, we may assume 
$$
q>2.
$$
The case when $q=2$ is indeed much more easier and follows from an easy modification of the case when $q>2$, and we would like to leave the details to the interested reader. 

We start with computing the ratio
\begin{equation} \label{20200826eq01}
\frac{\mu(J_p)}{\mu(J_{p-1})}
\end{equation} 
with the assumption when $Z$ is on the right hand side of $J$ and $Z \subsetneq J^\ell$. 

Recall that in the case when $J_p \subset I$, $K \in \{0, \dots, 2\alpha\}$ is the integer such that
$$
l(J_p) \in E^{(K)}.
$$
Moreover, we assume that $K' \in \{0, \dots, 2\alpha\}$ with $0 \le K+K' \le 2\alpha$ is the integer such that
$$
r(J_p) \in E^{(K')}. 
$$
We make a comment that the ratio \eqref{20200826eq01}, or in other words, the constant $A$ defined in assumption \eqref{20200825eq01}, is \emph{independent} of the choice of $K$ and $K'$. 

\subsection{Computation of $\frac{\mu(J_p)}{\mu(J_{p-1})}$ if $K$ does not exist.} \label{20200831case01}

First, we consider the case when $K$ does not exist, that is, $l(J_p) \notin I$. Note that the exhaustion procedure in this case is indeed trivial: all the $J_1, \dots, J_{p-1}$ do not intersect $I$ and hence 
$$
\mu (J_j)=|J_p|, \quad j=1, \dots, p-1. 
$$
There are several cases for $K'$.

\medskip

\subsubsection{$K'$ does not exist.} In this case, $J \cap I=\emptyset$, in particular, the ratio in \eqref{20200826eq01} takes the value $1$.

\medskip

\subsubsection{$K'=0$} In this case, the weight on $J_p$ is either $1$ or $a$, therefore, we have
$$
a|J_p| \le \mu(J_p) \le |J_p|, 
$$
and hence
$$
a\le \frac{\mu(J_p)}{\mu(J_{p-1})} \le 1. 
$$
In particular, one may pick $A=\frac{1}{a}$ in this case.

\medskip

\subsubsection{$K' \ge 1$.} Note that in this case, $E^{(0)} \subseteq J_p$, and we can always bound $\mu(J_p)$ from above as follows
\begin{eqnarray*}
\mu(J_p)%
&=& \mu \left( \left[l(J_p), l(I) \right] \right)+ \mu\left( \left[l(I), r(J_p) \right] \right) \\
&\le& |J_p|+\mu(I) = |J_p|+|I| \\
&\le& |J_p|+\frac{q}{q-2}\left|E^{(0)}\right|  \le |J_p|+\frac{q}{q-2}\left|J_p \right| \\
&=& \frac{2q-2}{q-2}|J_p|. 
\end{eqnarray*}

While for the lower bound, let us consider two different sub-cases.  
\begin{enumerate}

\item [$\bullet$] If $\left|\left[l(J_p), l(I) \right] \right| \le \frac{1}{2}|J_p|$, then $$
\mu(J_p) \ge \mu \left( \left[l(J_p), l(I) \right] \right)=\left|\left[l(J_p), l(I) \right] \right|>\frac{1}{2}|J_p|.
$$

\item [$\bullet$]  If $\left|\left[l(J_p), l(I) \right] \right| <\frac{1}{2}|J_p|$, then
\begin{eqnarray*}
\mu(J_p)%
&\ge& \mu \left(E^{(0)} \right) = a \left|E^{(0)} \right| \\
&=& \frac{a(q-2)}{q} |I| \ge \frac{a(q-2)}{q} \left| \left[l(I), r(J_p) \right] \right| \\
&\ge& \frac{a(q-2)}{2q} |J_p|.
\end{eqnarray*}

\end{enumerate}

Therefore, in this case we have
$$
\min \left\{  \frac{|J_p|}{2}, \frac{a(q-2)|J_p|}{2q} \right\} \le \mu(J_p) \le \frac{(2q-2)|J_p|}{q-2}, 
$$
and we can pick $A$ accordingly. 

\medskip

\subsection{Computation of $\frac{\mu(J_p)}{\mu(J_{p-1})}$ if $K$ exists.}

Note that if $K$ exists, then $K'$ also exists. Let us consider several possibilities. 

\subsubsection{$K'=0$.} \label{20200826sec01}

In this case, we have $J_p \subset E^{(K)}$. There are three possibilities in this situation:

\begin{enumerate}

\item [$(a)$.] If $\alpha \le K \le 2\alpha$, using Figure \ref{20200811Fig03}, we have $\mu(J_p)=a^{\alpha+1}b^{K-\alpha} |J_p|$. Moreover, we also have
$$
a^{\alpha+1}b^{K-\alpha-1}|J_p| \le \mu(J_{p-1}) \le  a^{\alpha+1}b^{K-\alpha}|J_p|, \quad \textrm{if} \ K>\alpha
$$
and
$$
a^{\alpha+1}|J_p| \le \mu\left(J_{p-1} \right) \le  a^{\alpha}|J_p|, \quad \textrm{if} \ K=\alpha. 
$$
Hence, we have
$$
a \le \frac{\mu(J_p)}{\mu(J_{p-1})} \le b
$$
and we can put $A=\max\left\{b, \frac{1}{a} \right\}$ in this case.

\medskip

\item [$(b)$.]  If $1 \le K \le \alpha-1$, using Figure \ref{20200811Fig02}, we have
$$
a^{K+1}|J_p| \le \mu(J_p) \le ba^{K}|J_p|.
$$
Moreover, we also have
$$
a^{K+1} |J_p| \le \mu(J_{p-1}) \le ba^K |J_p|. 
$$
Hence, we have 
$$
\frac{a}{b} \le \frac{\mu(J_p)}{\mu(J_{p-1})} \le \frac{b}{a}
$$
and we can let $A=\frac{b}{a}$ in this case. 

\medskip

\item [$(c)$.] If $K=0$, then $\mu(J_p)=a|J_p|$. Indeed, it is easy to see that in this case
$$
a|J_p| \le  \mu(J_{p-1}) \le |J_p|, 
$$
and hence
$$
a \le \frac{\mu(J_p)}{\mu(J_{p-1})} \le 1,
$$
which suggests that one can take $A=\frac{1}{a}$ in this case. 
\end{enumerate}

\medskip

\subsubsection{$K'=1$.} This case is very similar to the case in Section \ref{20200826sec01}, and the only difference here is the estimate of $\mu(J_p)$.  Note that in this case, $0 \le K<2\alpha$ since $K'=1$. There are again three possibilities in this situation:
\begin{enumerate}
    \item [(a).] If $K=2\alpha-1$, then using Figure \ref{20200811Fig01}, we have
    $$
    a^{\alpha+1} b^{\alpha-1} |J_p| \le \mu(J_p) \le a^\alpha b^\alpha |J_p|.
    $$
    While the estimate of $\mu(J_{p-1})$ is similar as before, namely, we have
    $$
    a^{\alpha+1} b^{\alpha-2} |J_p| \le \mu(J_{p-1}) \le a^{\alpha+1}b^{\alpha-1} |J_p|,
    $$
    and therefore
    $$
    1 \le \frac{\mu(J_p)}{\mu(J_{p-1})} \le \frac{b^2}{a},
    $$
    where we can put $A=\frac{b^2}{a}$ in this case;  
    
    \medskip
    
    \item [(b).] If $\alpha \le K \le 2\alpha-2$, then then using Figure \ref{20200811Fig03}, we have
    $$
    a^{\alpha+1}b^{K-\alpha}|J_p| \le \mu(J_p) \le a^{\alpha+1}b^{K+1-\alpha}|J_p|
    $$
    and
    $$
    a^{\alpha+1}b^{K-\alpha-1}|J_p| \le \mu(J_{p-1}) \le  a^{\alpha+1}b^{K-\alpha}|J_p|.
    $$
    Therefore,
    $$
    1 \le \frac{\mu(J_p)}{\mu(J_{p-1})} \le b^2,
    $$
    which suggests we can put $A=b^2$ in this case; 
    
    \medskip
  
    \item [(c).] If $0 \le K \le \alpha-1$, then using Figure \ref{20200811Fig01} and Figure \ref{20200811Fig02}, we have
    $$
    \begin{cases}
   a^{K+2} |J_p| \le \mu(J_p) \le ba^{K}|J_p|,, \quad \hfill K \ge 1; \\
\medskip \\
a^2 |J_p| \le \mu(J_p) \le ab |J_p|, \quad \hfill K=0, 
\end{cases} 
    $$
    and
    \begin{equation} \label{20200826eq21}
\begin{cases}
a^{K+1} |J_p| \le \mu(J_{p-1}) \le ba^K |J_p|, \quad \hfill K \ge 1; \\
\medskip \\
a|J_p| \le \mu(J_{p-1}) \le |J_p|, \quad \hfill K=0. 
\end{cases} 
    \end{equation} 
    This means that in this case we have
    $$
    \frac{a^2}{b} \le \frac{\mu(J_p)}{\mu(J_{p-1})} \le \frac{b}{a}
    $$
    and with $A$ being $\frac{b}{a^2}$ in this case. 
    
\end{enumerate}

\medskip

\subsubsection{$K'>1$. } In this case, we have 
$$
E^{(K+1)} \subset J_p \subset H^{(K)},
$$
where we identify $H^{(0)}:=\bigcup\limits_{k=0}^{2\alpha} E^{(k)}$. Note that 
$$
\frac{\left|E^{(K+1)}\right|}{\left|H^{(K)}\right|}=\frac{q-1}{q^2},
$$
which implies 
\begin{equation} \label{20200826eq05}
\frac{\left|E^{(K+1)}\right|}{\left|J_p\right|} \ge \frac{q-1}{q^2}. 
\end{equation}
Moreover, since $K’ \ge 2$, the defining condition on $K'$ implies $0 \le K \le 2\alpha-2$. There are again several cases. 

\begin{enumerate}
    \item [(a).] If $1 \le K+1 \le \alpha-1$, using \eqref{20200812eq01}, \eqref{20200812eq02} and \eqref{20200826eq05}, we can bound $\mu(J_p)$ from below as follows,
    \begin{eqnarray*}
    \mu \left(J_p\right)%
    &\ge& \mu \left( E^{(K+1)} \right)=(q-a) \cdot \frac{a^{K+1}|I|}{q^{K+2}} \\
    &=& \frac{(q-a)a^{K+1}}{q-1} \cdot \frac{q-1}{q^{K+2}} \cdot |I| = \frac{(q-a)a^{K+1}}{q-1} \cdot \left| E^{(K+1)} \right|\\
    & \ge& \frac{(q-a)a^{K+1}}{q-1} \cdot \frac{q-1}{q^2} |J_p| = \frac{(q-a)a^{K+1}}{q^2}  |J_p|.
    \end{eqnarray*}
    While for the upper bound, we have if $K>1$, 
    \begin{eqnarray} \label{20200827eq01}
    \mu(J_p)%
    &\le& \mu \left(H^{(K)} \right)=\frac{a^{K}|I|}{q^{K}} = \frac{a^{K}}{q^{K}} \cdot \frac{q^{K+2}}{q-1} \cdot \frac{q-1}{q^{K+2}} \cdot |I| \nonumber \\ 
    &=& \frac{a^{K} \cdot q^2}{q-1} \cdot \left| E^{(K+1)} \right| \le \frac{q^2 a^{K} }{q-1} \cdot \left| J_p \right|;
    \end{eqnarray}
    and if $K=0$, 
    \begin{eqnarray*}
    \mu(J_p)%
    &\le& \mu \left(H^{(0)} \right)=\frac{a(q-1)|I|}{q}=aq \cdot \frac{q-1}{q^2} |I| \\
    &=& aq \cdot \left|E^{(1)} \right| \le aq \cdot |J_p|. 
    \end{eqnarray*}
    The estimate for $\mu(J_p)$ in this case is exactly the same as \eqref{20200826eq21}, and hence
    $$
    \frac{(q-a)a}{bq^2} \le \frac{\mu(J_p)}{\mu(J_{p-1})} \le \frac{q^2}{a(q-1)}, 
    $$
    which implies that we can put
    $$
    A= \frac{bq^2}{a(q-a)}
    $$
    in this case. 
    \medskip
    
    \item [(b).] If $\alpha \le K+1 \le 2\alpha-1$, we can bound $\mu(J_p)$ from below by
    \begin{eqnarray*}
    \mu(J_p)%
    &\ge& \mu\left( E^{(K+1)} \right)=(q-b) \cdot \frac{b^{K+1-\alpha}a^\alpha |I|}{q^{K+2}} \\ 
    &=& \frac{(q-b) \cdot b^{K+1-\alpha}a^{\alpha}}{q-1} \cdot \frac{q-1}{q^{K+2}} \cdot |I| \\ 
    &=& \frac{(q-b) \cdot b^{K+1-\alpha}a^{\alpha}}{q-1} \cdot \left| E^{(K+1)} \right| \\
    &\ge& \frac{(q-b) \cdot b^{K+1-\alpha}a^{\alpha}}{q-1} \cdot \frac{q-1}{q^2} |J_p| \\
    &=& \frac{(q-b) \cdot b^{K+1-\alpha}a^{\alpha}}{q^2} \cdot |J_p|.
    \end{eqnarray*}

    While for the upper bound, we need to consider two sub-cases. If $K=\alpha-1$, then following the same argument in \eqref{20200827eq01}, we have 
    $$
    \mu(J_p) \le \frac{q^2 a^{\alpha-1}}{q-1} \cdot |J_p|.
    $$
    If $K \ge \alpha$, then 
    \begin{eqnarray*}
    \mu(J_p)%
    &\le& \mu \left( H^{(K)} \right)=\frac{a^\alpha b^{K-\alpha}|I|}{q^{K}} = \frac{a^\alpha b^{K-\alpha}}{q^{K}} \cdot \frac{q^{K+2}}{q-1} \cdot \frac{q-1}{q^{K+2}} \cdot |I| \\
    &=& \frac{a^\alpha b^{K-\alpha}}{q^{K}} \cdot \frac{q^{K+2}}{q-1} \cdot \left| E^{(K+1)} \right|\le \frac{q^2 a^\alpha b^{K-\alpha}}{q-1} |J_p|.
    \end{eqnarray*}
While for the estimate of $\mu(J_{p-1})$, we have
$$
\begin{cases}
a^{\alpha} |J_p| \le \mu(J_{p-1}) \le ba^{\alpha-1}|J_p|, \quad \hfill K=\alpha-1; \\
\medskip \\
a^{\alpha+1} b^{K-\alpha-1} |J_p| \le \mu(J_{p-1}) \le a^{\alpha+1} b^{K-\alpha} |J_p|, \quad \hfill K \ge \alpha.
\end{cases}
$$
All these estimates yield
$$
\frac{(q-b)a}{bq^2} \le \frac{\mu(J_p)}{\mu(J_{p-1})} \le \frac{bq^2}{a(q-1)}
$$
and 
$$
A=\frac{bq^2}{a(q-1)}
$$
in this case. 
\end{enumerate}

\medskip

Therefore, combining all the estimate of $A$, together with the exhaustion procedure, the proof for the case when $Z$ is on the right hand side of $J$ is complete. Now we turn to the other cases. 

\section{The analysis part IV: Other cases and the proof of Theorem \ref{main result 2}} 
\label{padic3}
In this section, we make some comments on how to adapt the exhaustion procedure to deal with the other two cases when 
\begin{enumerate}
    \item [(1).] $Z$ is on the left hand side of $J$;
    \item [(2).] $Z \in J$. 
\end{enumerate}
This allows us to conclude the measure $\mu$ construction in Section \ref{20200809sec01} is $p$-adic doubling, which proves Theorem \ref{main result}. Finally, we will extend our result to any finite collection of primes.

\subsection{$Z$ in on the left hand side of $J$.}  \label{20200815Sec01} Recall that $\Upsilon$ is on the right of $Z$. First we note that if $\Upsilon \in J$, then $\Upsilon$ has to be either $l(J)$ or $r(J)$, otherwise this will force $J$ to be $J^\ell$ or some $p$-adic ancestor of $J^\ell$, which contradicts the our assumption $J \subsetneq J^\ell$. Therefore, $\Upsilon$ is either located on the right hand side of $\Upsilon$ or on the left. 

The case when $\Upsilon$ is on the right hand side of $J$ is trivial. Since $J \subset \left[\textZeta, \Upsilon \right] \subset G^{(2\alpha)}$, and \eqref{20200809eq02} holds with the constant $1$ (since the weight on $G^{(2\alpha)}$ is $a^\alpha b^\alpha$). 

Hence, we may assume $\Upsilon$ is on the left hand side of $J$. However, by \eqref{20200722eq03}, we know that $\Upsilon-\textZeta \le \frac{|I|}{q^{100\alpha}}$, which is ``negligible" compared to the length of $G^{(2\alpha)}$, which is $\frac{|I|}{q^{2\alpha}}$. In other words, this motivates us to treat $\Upsilon$ and $\textZeta$ ``the same" under such a situation. Therefore, the proof of this case follows from an easy modification of the arguments presented for the case when $Z$ is on the right hand side of $J$ and we would like to leave the detail to the interested reader.

\subsection{$Z \in J$.}  The last case is also an application of the exhaustion procedure in Section \ref{exhaustion procedure}. To begin with, we note that since $Z \in J$, $\Upsilon$ is forced to located on the right hand side of $J$. Otherwise, $\Upsilon$ will be an interior point of $J$, and following the argument in the first paragraph in Section \ref{20200815Sec01},  this contradicts the assumption $J \subsetneq J^{\ell}$. 

We consider several possibilities.

\begin{enumerate}

\item [(1).] $l\left(H^{(2\alpha)} \right) \notin J_p$. Since $Z \in J$, it follows that $r(J)=r(J_p)>l\left(H^{(2\alpha)} \right)$. Let $K \in \{1, \dots, 2\alpha-1\}$ be the unique integer such that $ l\left(H^{(2\alpha)} \right) \in J_K$ (if such a $K$ does not exist, then $J \subset H^{(2\alpha)} \cap G^{(2\alpha)}$ and \eqref{20200809eq02} holds trivially with the constant $1$, since the weight on both $H^{(2\alpha)}$ and $G^{(2\alpha)}$ is $a^\alpha b^\alpha$). This means we have
$$
\mu \left(J_{K+1} \right)=\dots=\mu \left(J_p \right)=a^\alpha b^{\alpha} |J_p|, 
$$
while the estimate of $\mu(J_1), \dots, \mu(J_K)$ follows exactly the same as the exhaustion procedure, and \eqref{20200814eq01} holds in this case. 

\medskip
\item [(2).] $l\left(H^{(2\alpha)} \right) \in J_p$. To begin with, we first note that by an application of the exhaustion procedure, we can show that there exists an absolute constant $C'>0$, such that
$$
\frac{1}{C'} \le \frac{\mu(J_{j_1})}{\mu(J_{j_2})} \le C', \quad j_1, j_2 \in \{1, \dots, p-1\},
$$
since $Z$ is on the right of $J_{p-1}$. Therefore, it suffices to show that there exists an absolute constant $C''>0$, such that
$$
\frac{1}{C''} \le \frac{\mu(J_p)}{\mu(J_{p-1})} \le C''. 
$$
This will follows from an easy argument by examining whether the ratio
$$
\frac{\left| \left[l(J_p), l\left(H^{(2\alpha)} \right) \right] \right|}{|J_p|}
$$
makes a significant contribution (for example, whether it is greater than $1/2$) or not. We would like to leave the detail to the interested reader. 
\end{enumerate}

\begin{rem}
We note that this is the key part where Step $2\alpha$ is needed, as by stopping at Step $\alpha$, for very small $J_p$ lying almost entirely to the right of $\textZeta$, the ratio $|J_p|/|J_{p-1}|$ would be essentially $(b/a)^\alpha$.
\end{rem}

\medskip

\subsection{Proof of Theorem \ref{main result 2}.} To this end, we prove Theorem \ref{main result 2}. Recall that $\{p_1, \dots, p_M\}$ is a given finite collection of primes. Without loss of generality, we may assume $p_1$ is the smallest, otherwise, we just re-label all these numbers. 

The idea to prove the much more general result Theorem \ref{main result 2} is to apply our previous argument to each of the pairs
$$
(p_i, p_1), \quad  i=2, \dots, M,
$$
where $p_1$ plays the role of $q$ and $p_i$ plays the role of $p$. In particular, the analysis part remains unchanged. 

The only difference is the number theory part, and we can modify the argument as follows. Let 
$$
C(p_i, p_1), \  m(p_i, p_1), \quad i=2, \dots, M
$$
be the integers defined as in Proposition \ref{120200724prop01}. Let us define
$$
C\left( \{p_i\}_2^M, p_1 \right):= \max_{i=2, \dots, M} C(p_i, p_1) 
$$
and 
$$
m\left( \{p_i\}_2^M, p_1 \right):= \max_{i=2, \dots, M} m(p_i, p_1).
$$
Here are the modification we need. 

\begin{manualtheorem}{\ref{20200722thm01}'}\label{20200827thm01}
There exists a collection of $p_1$-adic intervals $\{I_\ell^{\alpha_\ell}\}_{\ell \ge 1}$ on $[0, 1)$, where $\alpha_\ell \ge 1$ is a positive integer associated to $\ell$, such that
\begin{enumerate}
    \item [(1).] For each $i \in \{2, \dots, M\}$, the collection of $p_i$-adic intervals $\{J^\ell_i\}_{\ell \ge 1}$ is pairwise disjoint and contained in $[0, 1)$, where $J_i^\ell$ is the smallest $p_i$-adic interval that contains $I_\ell^{\alpha_\ell}$. In particular, the collection $\{I_\ell^{\alpha_\ell}\}_{\ell \ge 1}$ is also pairwise disjoint;
    
    \medskip

    \item [(2).] For each $\alpha \ge 1, \alpha \in \N$, there are only finitely many $\ell \ge 1$, such that $\alpha_\ell=\alpha$. 
    
    \medskip
    
    \item [(3).] For each $\ell \ge 1$ and $i \in \{2, \dots, M\}$, 
    \begin{equation} \label{20200722eq03}
    0<\Upsilon\left(J_i^\ell \right)-\textZeta\left(I_\ell^{\alpha_\ell} \right) \le q^{-100\alpha_\ell} \left| I_\ell^{\alpha_\ell} \right|. 
    \end{equation}
\end{enumerate}
\end{manualtheorem}

\begin{manualproposition}{\ref{20200805prop01}'} \label{20200827prop01}
Given any interval $\widetilde{J} \subset [0, 1]$ ($\widetilde{J}$ is not necessarily $p_i$-adic, for any $i=2, \dots, M$) and any $\varepsilon>0$, there exists a $p_1$-adic interval $I \subset \widetilde{J}$ such that
$$
0<\Upsilon(J_i)-Z(I) \le \varepsilon |I|, \quad i=2, \dots, M,
$$
where $J_i$ is the smallest $p_i$-adic interval that contains $I$. 
\end{manualproposition}

The proof of Proposition \ref{20200827prop01} follows an easy modification of the one of Proposition \ref{20200805prop01} by replacing $C(p, q)$ by $C\left( \{p_i\}_2^M, p_1 \right)$, and $m(p, q)$  by $m\left( \{p_i\}_2^M, p_1 \right)$, respectively (Note that with such a modification, the counterpart of Proposition \ref{Prop arithmetic progressions} holds automatically). Moreover, for the proof of Theorem \ref{20200827thm01}, we may start with taking an infinite collection of $\{\widetilde{j}^\ell\}_{\ell \ge 1}$ of pairwise disjoint $\left(\prod_{i=2}^M p_i\right)$-adic sub-intervals on $[0, 1]$, and the rest are the same as those in Theorem \ref{20200722thm01}. We would like to leave the details to the interested reader.

\bigskip

\section{Applications}
\label{applications}

We now use our results to show an application related to the reverse H\"older inequality, mentioned in the introduction.  Though well equipped to do so by our earlier analysis, including the exhaustion procedure, the proofs require significant care.  To provide clarity, we explain our reasoning within the proofs before providing the detailed calculations.  We begin with some definitions.

Let $p$ and $q$ be a pair of primes with $p>q$ and $w$ be a weight (that is, a nonnegative locally integrable function). We may also assume $q>2$ as before, while the case $q=2$ follows from an easy modification of the proof for $q>2$.

Let further, $w_\mu$ be the weight associated to the measure $\mu$ that we have constructed in Section \ref{20200809sec01}, that is
$$
\mu(I)=\int_I w_\mu dx, \quad \textrm{for any interval} \ I. 
$$

Define the \emph{reverse H\"older and q-adic reverse H\"older classes} as follows:
\begin{defn}
Let $r>1$. We say that $w\in RH_r$ if 
\begin{equation}
\label{20200829eq01}
\left(\fint_I w^r \right)^{\frac{1}{r}} \leq C \fint_I w 
\end{equation}
for all intervals $I$, where $C$ is an absolute constant. Moreover, we say $w \in RH_1$ if $w \in RH_r$ for some $r>1$, that is
$$
RH_1:=\bigcup_{r>1} RH_r. 
$$
\end{defn}
\begin{defn}
Let $r>1$. We say that $w\in RH_r^q$ if
\begin{equation} \label{20200829eq02}
\left(\fint_Q w^r \right)^{\frac{1}{r}} \leq C \fint_Q w
\end{equation} 
for all $q$-adic intervals $Q$, where $C$ is an absolute constant and $w$ is $q$-adic doubling. Moreover, we say $w \in RH^q_1$ if $w \in RH_r^q$ for some $r>1$, namely
$$
RH^q_1:=\bigcup_{r>1} RH_r^q. 
$$
\end{defn}

Note that it is well-known that any $RH_r$ weight is doubling, but a weight which satisfies \eqref{20200829eq02} is not necessarily $q$-adic doubling (see, for example \cite{LPW}).  This is why the $q$-adic doubling assumption is added to the second definition.  

The study of these weights has been extensive, more information and some recent applications can be found in \cite{CN}, \cite{P2}, \cite{PWX}, \cite{LPW}, \cite{HPR}, \cite{K et al}, \cite{AW}  among many others.   There is also interesting complimentary work done in \cite{DCU -1}.  Reverse H\"older weights are relevant in the theory of quasiconformal maps, which was the original motivation for their study \cite{G}. 

Consider the $w = w_\mu$ from our construction.  Since this $w_\mu$ is not doubling, then it does not satisfy \eqref{20200829eq01}, namely $w_\mu \notin RH_r$.  

\begin{prop} \label{20200831prop02}
The weight $w_\mu \in RH_1^q$.
\end{prop}

\begin{proof}
It suffices to show that $w \in RH_r^q$ for some $r>1$. Let us fix an $r$ with 
$$
1<r<\frac{\ln q}{\ln b},
$$
where we recall that $b<q$ from the construction in Section \ref{analysis 1}. We denote
$$
B_1:=\frac{b^r}{q} \quad \textrm{and} \quad  B_2:=\frac{a^r}{q},
$$
which by our assumption, clearly satisfies $0<B_1, B_2<1$. 

Let $I$ be any $q$-adic interval. Without loss of generality, we may assume $I$ intersects at least one $I_\ell^{\alpha_\ell}$, otherwise the constant $C$ in \eqref{20200829eq02} is simply $1$. 

\textbf{Particular case.} First of all, we consider the case when $I \subseteq I_\ell^{\alpha_\ell}$ for some $\ell$. Note that among all the $q$-adic offspring of $I_\ell^{\alpha_\ell}$, the only interesting cases are $I$ coincides one of the following: 
\begin{equation} \label{20200829eq03}
I_\ell^{\alpha_\ell}, H^{(k)} \ \textrm{and} \ G^{(k)}, \quad k=1, \dots, 2\alpha-1. 
\end{equation} 
Otherwise, the weight on $I$ is of the form $a^x b^y$ for some $x, y \in \N$, and in this case, one can easily check that
$$
\left(\fint_I w_\mu^r \right)^{\frac{1}{r}}=\fint_I w_\mu, 
$$
that is, the constant $C$ in \eqref{20200829eq02} is $1$. 

Let us consider five different cases for the intervals in \eqref{20200829eq03}.

\medskip

\textit{Case I: $I=H^{(k)}, k=\alpha, \dots, 2\alpha-1$.} On one hand
$$
\fint_{H^{(k)}} w_\mu = \frac{\mu(H^{(k)})}{|H^{(k)}|}=a^\alpha b^{k-\alpha},
$$
and on the other hand
\begin{eqnarray*}
\fint_{H^{(k)}} w_\mu^r%
&=& \frac{a^{\alpha r} b^{\alpha r}}{q^{2\alpha-k}}+ \frac{a^{(\alpha+1)r} b^{(k-\alpha)r} \cdot (q-1)}{q} \cdot \sum_{i=0}^{2\alpha-k-1} \left(\frac{b^{r}}{q}\right)^i\\
& \le & a^{\alpha r} b^{(k-\alpha)r} \cdot \left[ B_1^{2\alpha-k}+ \frac{q-1}{q} \cdot \frac{a^r}{1-B_1} \right] \\
& \le& a^{\alpha r} b^{(k-\alpha)r } \cdot \left[ 1+ \frac{q-1}{q} \cdot \frac{1}{1-B_1} \right]  \\
&=&  C_1 \cdot a^{\alpha r} b^{(k-\alpha)r },
\end{eqnarray*}
where we denote 
$$
C_1:=1+ \frac{q-1}{q} \cdot \frac{1}{1-B_1}.
$$
Therefore, \eqref{20200829eq02} holds in this case with the constant $C=C_1^{\frac{1}{r}}$.

\medskip

\textit{Case II: $I=H^{(k)}, k=1, \dots, \alpha-1$.} On one hand, we have
$$
\fint_{H^{(k)}} w_\mu = \frac{\mu(H^{(k)})}{|H^{(k)}|}=a^k.
$$
On the other hand, 
\begin{eqnarray*}
\fint_{H^{(k)}} w_\mu^r%
&=& \frac{a^{\alpha r} b^{\alpha r}}{q^{2\alpha-k}}+\frac{a^{(\alpha+1)r} (q-1)}{q^{\alpha-k+1}} \cdot  \left[ \sum_{i=0}^{\alpha-1} \left(\frac{b^r}{q} \right)^i \right] \\
&& \quad \quad + \frac{b^ra^{kr}}{q} \cdot \left[ \sum_{i=0}^{\alpha-k-1} \left( \frac{a^r}{q} \right)^i \right] \\
&& \quad \quad +\frac{a^{(k+1)r}(q-2)}{q} \cdot \left[ \sum_{i=0}^{\alpha-k-1} \left( \frac{a^r}{q} \right)^i \right] \\
&\le&  \frac{a^{\alpha r} b^{\alpha r}}{q^{2\alpha-k}}+\frac{a^{(\alpha+1)r} (q-1)}{q^{\alpha-k+1}} \cdot \frac{1}{1-\frac{b^r}{q}}+\frac{b^r a^{kr}}{q} \cdot \frac{1}{1-\frac{a^r}{q}} \\
&& \quad \quad +\frac{a^{(k+1)r}(q-2)}{q} \cdot \frac{1}{1-\frac{a^r}{q}} \\
&=& a^{kr} \left( B_1^\alpha B_2^{\alpha-k}+ \frac{B_2^{\alpha-k+1} (q-1)}{1-B_1}+ \frac{B_1+B_2(q-2)}{1-B_2} \right) \\
& \le& a^{kr} \left(1+\frac{q-1}{1-B_1}+\frac{B_1+B_2(q-2)}{1-B_2} \right) \\
&& \quad \quad (\textrm{since} \ 0<B_1, B_2<1.) \\ 
&=& C_2 \cdot a^{kr}, 
\end{eqnarray*}
where we denote 
$$
C_2:=1+\frac{q-1}{1-B_1}+\frac{B_1+B_2(q-2)}{1-B_2}
$$
and therefore, \eqref{20200829eq02} holds in this case with the constant $C=C_2^{\frac{1}{r}}$.

\medskip

\textit{Case III: $I=G^{(k)}, k=\alpha, \dots, 2\alpha-1$.} On one hand, we have
$$
\fint_{G^{(k)}} w_\mu = \frac{\mu(G^{(k)})}{|G^{(k)}|}=b^\alpha a^{k-\alpha}.
$$
On the other hand,
\begin{eqnarray*}
\fint_{G^{(k)}} w^r_\mu%
&=& \frac{b^{\alpha r} a^{\alpha r}}{q^{2\alpha-k}}+ b^{\alpha r} a^{r(k-\alpha)} (q-2) \cdot \frac{a^r}{q} \cdot \sum_{i=0}^{2\alpha-k-1} \left(\frac{a^r}{q} \right)^i \\
&& \quad \quad + b^{\alpha r} a^{r(k-\alpha)} \cdot \frac{b^r}{q} \cdot \sum_{i=0}^{2\alpha-k-1} \left(\frac{a^r}{q} \right)^i \\
&\le & b^{\alpha r} a^{(k-\alpha)r} \left[ \left(B_2 \right)^{2\alpha-k}+ \frac{(q-2)B_2+B_1}{1-B_2} \right] \\
&\le& b^{\alpha r} a^{(k-\alpha)r} \left[ 1+ \frac{(q-2)B_2+B_1}{1-B_2} \right] \\
&=& C_3 \cdot b^{\alpha r} a^{(k-\alpha)r},
\end{eqnarray*}
where we denote
$$
C_3:= 1+ \frac{(q-2)B_2+B_1}{1-B_2}
$$
and therefore, \eqref{20200829eq02} holds in this case with the constant $C=C_3^{\frac{1}{r}}$.

\medskip

\textit{Case IV: $I=G^{(k)}, k=1, \dots, \alpha-1$.}  On one hand, we have
$$
\fint_{G^{(k)}} w_\mu = \frac{\mu(G^{(k)})}{|G^{(k)}|}=b^k.
$$
On the other hand,
\begin{eqnarray*}
\fint_{G^{(k)}} w^r_\mu%
&=& \frac{b^{\alpha r} a^{\alpha r}}{q^{2\alpha-k}}+ b^{kr} \cdot (q-2) \cdot \left(\frac{b^r}{q} \right)^{\alpha-k} \cdot \frac{a^r}{q} \sum_{i=0}^{\alpha-1} \left(\frac{a^r}{q} \right)^{i-1} \\
&& \quad \quad + b^{kr} \cdot \frac{b^{(\alpha+1-k)r}}{q^{\alpha+1-k}} \cdot \sum_{i=0}^{\alpha-1} \left( \frac{a^r}{q} \right)^i \\
&& \quad \quad + b^{kr} \cdot\frac{a^r}{q} \cdot (q-1) \cdot \sum_{i=0}^{\alpha-k-1} \left( \frac{b^r}{q} \right)^i \\
&\le& b^{kr} \left[B_1^{\alpha-k} B_2^\alpha+\frac{(q-2)B_2+B_1}{1-B_2}+\frac{(q-1)B_2}{1-B_1} \right] \\
&\le& b^{kr} \left[1+\frac{(q-2)B_2+B_1}{1-B_2}+\frac{(q-1)B_2}{1-B_1} \right] \\
&& \quad \quad (\textrm{since} \ 0<B_1, B_2<1.) \\ 
&=& C_4 \cdot b^{kr}, 
\end{eqnarray*}
where we denote
$$
C_4:= 1+\frac{(q-2)B_2+B_1}{1-B_2}+\frac{(q-1)B_2}{1-B_1}
$$
and therefore, \eqref{20200829eq02} holds in this case with the constant $C=C_4^{\frac{1}{r}}$.

\medskip

\textit{Case V: $I=I_\ell^{\alpha_\ell}$.} One one hand, we have
$$
\fint_{I_\ell^{\alpha_\ell}} w_u=\frac{\mu(I_\ell^{\alpha_\ell})}{|I_\ell^{\alpha_\ell}|}=1. 
$$
Other the other hand
\begin{eqnarray*}
\fint_{I_\ell^{\alpha_\ell}} w_u^r%
&=& \frac{q-2}{q} \cdot \fint_{E^{(0)}} w_u^r+ \frac{1}{q} \cdot \fint_{H^{(1)}} w_u^r+ \frac{1}{q} \cdot \fint_{G^{(1)}} w_u^r \\
&\le&  \frac{(q-2)a^r}{q}+ \frac{C_2 a^r}{q}+\frac{C_4 b^r}{q} \\
&=& (q-2+C_2) B_2+ C_4B_1, 
\end{eqnarray*}
where in the second to last line, we use the estimate from \textit{Case II} and \textit{Case IV} above. Hence in this case, \eqref{20200829eq02} holds with the constant $C=C_5^{\frac{1}{r}}$, where we denote
$$
C_5:=(q-2+C_2) B_2+ C_4B_1.
$$

\textbf{General case.} Finally, we consider the case when $I$ contains one or more $I_\ell^{\alpha}$. By our construction, it is clear that
$$
\fint_I w_\mu=1. 
$$
Let us assume there are indices $\ell_1, \dots, \ell_{\widetilde{M}}$, such that
$$
I_{\ell_j}^{\alpha_{\ell_j}} \subsetneq I, \quad j=1, \dots, \widetilde{M}. 
$$
Without loss of generality, we may assume $\widetilde{M}<\infty$. The key point is that the constant we get here is independent of the choice of any finite $\widetilde{M}$, therefore, the estimate for the case when $I$ contains infinitely many $I_\ell^{\alpha_\ell}$ follows by a standard limiting argument.

From the proof of the case when $I \subset I_\ell^{\alpha_\ell}$, we see that
$$
\left( \fint_{I_{\ell_j}^{\alpha_{\ell_j}}} w_u^r  \right)^{\frac{1}{r}} \le C' \fint_{I_{\ell_j}^{\alpha_{\ell_j}}} w_u, \quad j=1, \dots, \widetilde{M}
$$
where we set
$$
C':= \max\left\{C_1^{\frac{1}{r}}, \dots, C_5^{\frac{1}{r}} \right\}.
$$
Let us assume
$$
\omega_j:=\frac{\left|I_{\ell_j}^{\alpha_{\ell_j}}\right|}{|I|}, \quad j=1, \dots, \widetilde{M}. 
$$
In particular, this suggests that
$$
\sum_{j=1}^{\widetilde{M}} w_j \le 1. 
$$
Finally, we denote
$$
I^c:=I \backslash \left( \bigcup_{j=1}^{\widetilde{M}} I_{\ell_j}^{\alpha_{\ell_j}} \right), \quad \textrm{and} \quad \omega_{\widetilde{M}+1}=1-\sum_{j=1}^{\widetilde{M}} w_j. 
$$
Note that $\omega_{\widetilde{M}+1}=\frac{|I^c|}{|I|}$.

Therefore, we have
\begin{eqnarray*}
\fint_I w_\mu^r%
&=& \frac{1}{|I|} \cdot \int_I w_\mu^r = \frac{1}{|I|} \cdot \left( \sum_{j=1}^{\widetilde{M}} \int_{I_{\ell_j}^{\alpha_{\ell_j}}} w_\mu^r+ \int_{I^c} w_\mu^r \right) \\
&=& \sum_{j=1}^{\widetilde{M}} \omega_j \fint_{I_{\ell_j}^{\alpha_{\ell_j}}} w_\mu^r+ \omega_{\widetilde{M}+1} \fint_{I^c} w_\mu^r \\
&\le&  \sum_{j=1}^{\widetilde{M}} \omega_j \cdot \left(C' \right)^r \cdot \left(\fint_{I_{\ell_j}^{\alpha_{\ell_j}}} w_\mu \right)^r+\omega_{\widetilde{M}+1} \\
&=& \left(C' \right)^r \cdot  \sum_{j=1}^{\widetilde{M}} \omega_j+\omega_{\widetilde{M}+1} \\
&\le& (C')^r+1
\end{eqnarray*}
that is, in general, we have
$$
\left( \fint_I w_\mu^r \right)^\frac{1}{r} \le C \fint_I w_\mu, \quad  \textrm{for  any $q$-adic interval $I$},
$$
with $C^r:= (C')^r+1$ if $1<r<\frac{\ln q}{\ln b}$. The proof is complete. 
\end{proof}

\begin{prop} \label{20200831prop01}
The weight $w_\mu \in RH_1^p$. 
\end{prop}

\begin{proof}
This proof of this proposition is an application of the exhaustion procedure and Proposition \ref{20200831prop02}.  We reference the structure and setup in Sections \ref{analysis 1} and \ref{padic1} often.

Recall that for each $I_\ell^{\alpha_\ell}$, $J_\ell$ is the smallest $p$-adic interval that contains $I_\ell^{\alpha_\ell}$ with $J_\ell \subset [0, 1)$ and the $J_\ell$'s are pairwise disjoint. 

Let $J$ be the $p$-adic interval to be tested. There are several reductions that we can make.

\medskip

\textbf{Reduction I:} To begin with, we may assume again that $J$ coincides with some $J_\ell$ or some of its $p$-adic offspring. Otherwise, if $J$ contains some $J_\ell$ properly or contains more than two $J_\ell$, we can argue as we did for Proposition \ref{20200831prop02}. 

\medskip

\textbf{Reduction II:} We may assume $\alpha=\alpha_\ell>2N$ where we recall that $N=\floor*{\frac{\log p}{\log q}}+1$ is constant we fixed at the beginning of the exhaustion procedure. Otherwise, we can simply estimate \eqref{20200829eq01} crudely, as all the weights here only depend on $a$, $b$, $q$ and $p$.  

\medskip

\textbf{Reduction III:} The third reduction would be that we may assume 
$$
J \cap \left(J_\ell \right)_1 \neq \emptyset \quad \textrm{or} \quad  J \cap \left(J_\ell \right)_2 \neq \emptyset,
$$
otherwise, the measure equipped on $J$ is the standard Lebesgue measure and \eqref{20200829eq01} holds trivially. Without loss of generality, let us assume 
the intersection of $J$ and $\left(J_\ell \right)_1$ is nonempty, the other case can be argued similarly as the ``mirror symmetric" argument in Section \ref{20200815Sec01}. 

\medskip

\textbf{Reduction IV:} Recall the points $\textZeta$ and $\Upsilon$ defined in Section \ref{analysis 1}. Note that it suffices to consider the case when $Z>l(J_p)$. Otherwise, by Reduction III, we have 
$$
J_p \subset [Z, \Upsilon],
$$
in particular, 
$$
|J_p| \le \left| [Z, \Upsilon] \right|<\frac{|I|}{q^{100\alpha}}.
$$
where $I=I_\ell^{\alpha_\ell}$. This implies that $J \subset H^{(2\alpha)} \cup G^{(2\alpha)}$ since $\left|H^{(2\alpha)}\right|=\left|G^{(2\alpha)} \right|=\frac{|I|}{q^{2\alpha}}$. Therefore, the estimate \eqref{20200829eq01} is again trivial with the constant $C=1$.

\medskip

\textbf{Reduction V:} Furthermore, we may assume that $J_p \subset I$, that is, there exists some $K \in \{1, \dots, 2\alpha\}$, such that $l(J_p) \in E^{(K)}$. Again, the case when such a $K$ does not exist is even easier.
\medskip

We first bound the term 
$$
\Theta_1:=\left(\fint_J w_\mu^r \right)^{\frac{1}{r}}
$$
from above.

Here comes the key observation: if $l(J_p) \in E^{(K)}$, then by the exhaustion procedure, $J \in H^{(K-N)}$, and we may assume $J \subseteq H^{(K-N')}$ ($N' \in \{0, 1, \dots, N\}$), where $H^{(K-N')}$ is the shortest $H^{(k)}$-interval that contains $J$.

\medskip

\textbf{Reduction VI:} Without loss of generality, we may assume $K-N' \ge 1$. Otherwise, we let $J'$ be the unique $p$-adic child which contains $l\left(H^{(1)} \right)$, and group the $p$-adic children of $J$ as follows:
\begin{enumerate}
  
 \item [(1).] $\left\{J_i: J_i \ \textrm{is on the left of} \ J'\right\} \cup \{J'\}$;
    
    \medskip
    
    \item [(2).] $\left\{J_i, J_i \ \textrm{is on the right of}  \ J'\right\}$.
    
\end{enumerate}

The first group either is weighted by $1$ or $a$, and the second group can be dealt with it by using the argument for $K-N' \ge 1$. Finally, we glue both groups together, and we may argue again as in the proof of the general case in Proposition \ref{20200831prop02}. 

\bigskip

Therefore, we have
\begin{equation} \label{20200901eq01}
\Theta_1^r = \fint_J w_\mu^r=\frac{1}{|J|} \int_J w_\mu^r \le \frac{1}{|J|} \int_{H^{(K-N')}} w_\mu^r
\end{equation}

\medskip

On the other hand, let 
$$
\Theta_2:=\fint_J w_\mu,
$$
and we would like to bound it from below. To begin with, we denote $J^* $ to be the rightmost $p$-adic children among the set $\{J_1, \dots, J_{p-1}\}$ such that there is only one value assigned to the weight on $J^*$, that is $d\left(\mu \big |_{J^*}\right)$ can be written as $\left(a^{n_1} b^{n_2} \right) dx$ for some $n_1, n_2 \in \{0, \dots, \alpha\}$. Note that by the geometric growth of the $E^{(k)}$'s , $J^*$ is either $J_{p-1}$, $J_{p-2}$ or $J_{p-3}$ (see, Figure \ref{20200904Fig01} for examples). This suggests us to treat $J^*$ as a very small shift of $J_p$, however, with a much easier expression to work with. 

    \begin{figure}[ht]
\begin{tikzpicture}[scale=4.5]
\draw (0.1, 0)--(1.1, 0);
\fill (0.1, 0) circle [radius=.3pt];
\fill (1.1, 0) circle [radius=.3pt];
\fill[red] (.35, 0) circle [radius=.4pt];
\fill (.9, 0) circle [radius=.2pt];
\fill (.8, -.01) node [below] {\tiny {\color{blue} $J_p$}};
\fill (.6, -.01) node [below] {\tiny {\color{blue} $J_{p-1}$}};
\fill (.6, .01) node [above] {\footnotesize $J^*$};
\draw [line width=0.5mm, red ] (.5, 0) -- (.7, 0); 
\fill (.5, 0) circle [radius=.2pt];
\fill (.7, 0) circle [radius=.2pt];
\draw [decorate,decoration={brace,amplitude=8pt,raise=10pt},yshift=2pt] (0.1, 0) -- (1.1, 0) node [black,midway,xshift=0cm, yshift=.9cm]  {\footnotesize $E^{(K)}$};
\draw (1.4, 0)--(2.4, 0);
\fill (1.4, 0) circle [radius=.3pt];
\fill (2.4, 0) circle [radius=.3pt];
\fill[red] (1.65, 0) circle [radius=.4pt];
\draw [decorate,decoration={brace,amplitude=8pt,raise=10pt},yshift=2pt] (1.4, 0) -- (2.4, 0) node [black,midway,xshift=0cm, yshift=.9cm]  {\footnotesize $E^{(K)}$};
\fill (1.75, 0) circle [radius=.2pt];
\fill (1.9, 0) circle [radius=.2pt];
\fill (1.825, -.01) node [below] {\tiny {\color{blue} $J_p$}};
\fill (1.675, -.01) node [below] {\tiny {\color{blue} $J_{p-1}$}};
\fill (1.525, -.01) node [below] {\tiny {\color{blue} $J_{p-2}$}};
\fill (1.525, .01) node [above] {\footnotesize $J^*$};
\draw [line width=0.5mm, red ] (1.45, 0) -- (1.6, 0); 
\fill (1.45, 0) circle [radius=.2pt]; 
\fill (1.6, 0) circle [radius=.2pt];
\draw (0.45, -.45)--(2.2, -.45);
\fill (1.2, -.45) circle [radius=.3pt];
\fill (2.2, -.45) circle [radius=.3pt];
\fill[red] (1.45, -.45) circle [radius=.4pt];
\draw [decorate,decoration={brace,amplitude=8pt,raise=10pt},yshift=-2pt] (1.65, -.45) -- (1.35, -.45) node [black,midway,xshift=0cm, yshift=-.9cm] {\tiny {{\color{blue} $J_{p-1}$}}};
\draw [decorate,decoration={brace,amplitude=8pt,raise=10pt},yshift=-2pt] (1.95, -.45) -- (1.65, -.45) node [black,midway,xshift=0cm, yshift=-.9cm] {\tiny {{\color{blue} $J_p$}}};
\draw [decorate,decoration={brace,amplitude=8pt,raise=10pt},yshift=-2pt] (1.35, -.45) -- (1.05, -.45) node [black,midway,xshift=0cm, yshift=-.9cm] {\tiny {{\color{blue} $J_{p-2}$}}};
\draw [decorate,decoration={brace,amplitude=8pt,raise=10pt},yshift=-2pt] (1.05, -.45) -- (.75, -.45) node [black,midway,xshift=0cm, yshift=-.9cm] {\tiny {{\color{blue} $J_{p-3}$}}};
\draw [line width=0.5mm, red] (1.05, -.45) -- (.75, -.45); 
\fill (.9, -0.44) node [above] {\footnotesize $J^*$}; 
\fill (.75, -.45) circle [radius=.2pt]; 
\fill (1.05, -.45) circle [radius=.2pt];
\fill (1.35, -.45) circle [radius=.2pt];
\fill (1.65, -.45) circle [radius=.2pt]; 
\fill (1.95, -.45) circle [radius=.2pt]; 
\draw [decorate,decoration={brace,amplitude=8pt,raise=10pt},yshift=2pt] (0.45, -.45) -- (1.2, -.45) node [black,midway,xshift=0cm, yshift=.9cm]  {\footnotesize $E^{(K-1)}$};
\draw [decorate,decoration={brace,amplitude=8pt,raise=10pt},yshift=2pt] (1.2, -.45) -- (2.2, -.45) node [black,midway,xshift=0cm, yshift=.9cm]  {\footnotesize $E^{(K)}$};
\end{tikzpicture}
\caption{Examples of all three possibilities of $J^*$ with $E^{(K)}$ if $1 \le K \le \alpha-1$, where we recall from Figure \ref{20200811Fig02} that the weight on the right hand side of the red point is $a^{K+1}$, and is $b a^K$ on the left hand side, respectively. Moreover, the weight on the part of $E^{(K-1)}$ which is adjacent to $E^{(K)}$ is $a^{K}$.}
\label{20200904Fig01}
\end{figure}
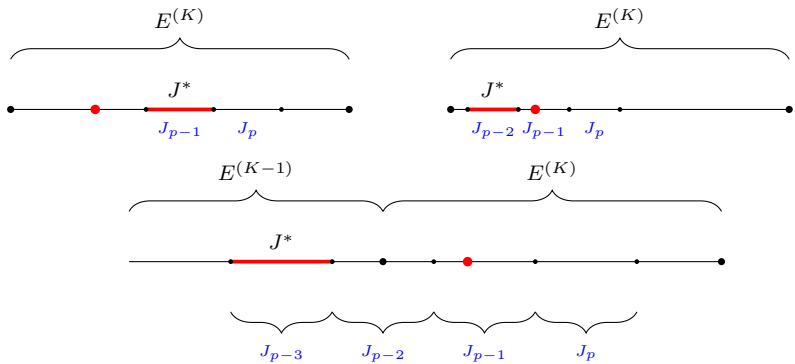

Therefore, we have
\begin{eqnarray} \label{20200901eq02}
\Theta_2%
&=& \frac{1}{|J|} \int_J w_\mu=\frac{1}{|J|} \cdot \left( \mu(J_1)+\dots+\mu(J_p) \right) \nonumber \\
&=& \frac{1}{|J|} \left( \frac{\mu(J_1)}{\mu(J^*)}+ \dots+\frac{\mu(J_p)}{\mu(J^*)} \right) \cdot \mu(J^*) \nonumber \\
& \ge& \frac{p}{\widetilde{C} |J|} \cdot \mu(J^*),
\end{eqnarray}
where in the last estimate, we use the exhaustion procedure and we assume the constant there is $\widetilde{C}$.

\medskip

Therefore, to prove the estimate \eqref{20200829eq01} with respect to the $p$-adic interval $J$, it suffices to show that
$$
\frac{\Theta_1}{\Theta_2} \le \frac{\widetilde{C}}{p} \cdot \left( \frac{\left|H^{(K-N')} \right|}{|J|} \right)^{\frac{1}{r}} \cdot \left( \fint_{H^{(K-N')}} w_\mu^r \right)^{\frac{1}{r}} \cdot \frac{|J|}{\mu(J^*)}
$$
is bounded above by some absolute constant, where in the above estimate, we use \eqref{20200901eq01} and \eqref{20200901eq02}. To this end, let us denote
$$
\Theta_3:=\left( \frac{\left|H^{(K-N')} \right|}{|J|} \right)^{\frac{1}{r}} \cdot \left( \fint_{H^{(K-N')}} w_\mu^r \right)^{\frac{1}{r}} \cdot \frac{|J|}{\mu(J^*)}.
$$
Let us start with analysing the term 
$$
\frac{\left|H^{(K-N')} \right|}{|J|}.
$$
Note that for most cases, $H^{(K-N')}$ is the $q$-adic interval which essentially has the ``same" size of $J$, priorly, we should expect
\begin{equation} \label{20200901eq10}
\frac{\left|H^{(K-N')} \right|}{|J|} \simeq 1, 
\end{equation} 
where the implicit constant here only depends on $q$. More precisely, we consider several possibilities.
\begin{enumerate}
    \item [(i).] If $N'\geq2$, then since $J \subseteq H^{(K-N')}$, we also have $E^{(K-N'+1)} \subset J$ (by the geometric structure, since $l(J_p) \in E^{(K)}$, otherwise $J$ would have to be contained in $H^{(K-1)}$, contradicting the choice of $N'$).  Therefore, we get the desired equation \eqref{20200901eq10}, where the implicit constant in the above equation only depends on $q$; 
    
    \medskip
    
    \item [(ii).] If $N'=0$ or $1$ and $J_p \subset E^{(K)} \cap E^{(K+1)}$, then the weight on $J$ will be one of the following situations: 
    \begin{enumerate}
    
    \item [$\bullet$] If $1 \le K \le \alpha-2$, then these weights are
    $$
    ba^{K-1}, a^K, ba^K, a^{K+1}, ba^{K+1}, a^{K+2}; 
    $$
    
    \medskip
    
        \item [$\bullet$] If $ \alpha+1 \le K \le 2\alpha$, then these weights are
    $$
    a^{\alpha+1}b^{K-1-\alpha}, a^{\alpha+1} b^{K-\alpha}, a^{\alpha+1} b^{K+1-\alpha};
    $$
    
    \medskip
    
    \item [$\bullet$] If $K=\alpha$, then these weights are 
    $$
    ba^{\alpha-1}, a^\alpha, a^{\alpha+1}, a^{\alpha+1}b;
    $$
    
    \medskip
    
    \item [$\bullet$] If $K=\alpha-1$, then these weights are
    $$
    ba^{\alpha-2}, a^{\alpha-1}, ba^{\alpha-1}, a^{\alpha}, a^{\alpha+1}. 
    $$
  
    \end{enumerate}
      
    The key point here is that for each of the situations above, the weights are comparable, in the sense the ratios between them are bounded above and below by some constant only depending on $a$ and $b$, independent of the choices of $K$ and $\alpha$. This allows us to establish the estimate \eqref{20200829eq01};
    \medskip
    
\item [(iii).] If $N'=0$ or $1$ and $E^{(K+1)} \subset J_p$, then since $E^{(K+1)}$ takes a large portion of either $H^{(K)}$ or $H^{(K-1)}$, we can conclude again that \eqref{20200901eq10} holds, with the implicit constant there depending only on $q$. 
\end{enumerate}
 From now on, we assume the ratio in \eqref{20200901eq10} is approximately $1$. 
\medskip

\textbf{Completing the argument:}

\medskip

Now we turn to estimate the rest two terms in $\Theta_3$, that is, 
$$
\left( \fint_{H^{(K-N')}} w_\mu^r \right)^{\frac{1}{r}} \quad \textrm{and} \quad  \frac{|J|}{\mu(J^*)}.
$$
Our goal is to show that
\begin{equation} \label{202009001eq11}
\left( \fint_{H^{(K-N')}} w_\mu^r \right)^{\frac{1}{r}} \cdot  \frac{|J|}{\mu(J^*)} \lesssim 1,
\end{equation}
where the implicit constant above should only depend on $a$, $b$, $p$ and $q$, and independent of $\alpha$ and $K$. 

We again need to consider several different cases. 

\textit{Case $I$: $1 \le K \le \alpha-1$. } First of all, we note that according to the choice of $J^*$, the value of $\mu(J^*)$ is 
$$
\textrm{either} \quad a^K|J^*|, \quad ba^K|J^*|, \quad \textrm{or} \quad a^{K+1}|J^*|.
$$
Since these three quantities differ by a constant multiple which only depends on $a$ and $b$, we may write
$$
\mu(J^*) \simeq a^K |J^*|.
$$
This means
$$
\frac{|J|}{\mu(J^*)}=\frac{p|J^*|}{\mu(J^*)} \simeq \frac{1}{a^K}.
$$
On the other hand, recall that in this case we have $1 \le K-N' \le \alpha-1$, by the computation in \textit{Case II} of Proposition \ref{20200831prop02}, we have
\begin{equation} \label{20200901eq12}
\fint_{H^{(K-N')}} w_\mu^r \le C_2 \cdot a^{(K-N')r} 
\end{equation} 
Therefore
$$
\left( \fint_{H^{(K-N')}} w_\mu^r \right)^{\frac{1}{r}} \cdot  \frac{|J|}{\mu(J^*)} \lesssim a^{K-N'} \cdot \frac{1}{a^K} \le \frac{1}{a^N},
$$
which implies the estimate \eqref{202009001eq11}.

\medskip

\textit{Case II: $\alpha \le K \le 2\alpha$.} Again, we start with estimating $\mu(J^*)$, and as before, we collect all the possible values of $\mu(J^*)$ and this gives us
$$
\mu(J^*) \simeq a^{\alpha} b^{K-\alpha} |J^*|,
$$
where the implicit constant only depends on $a$ and $b$. This means
$$
\frac{|J|}{\mu(J^*)}=\frac{p|J^*|}{\mu(J^*)} \simeq \frac{1}{a^{\alpha} b^{K-\alpha}}. 
$$
While for the upper bound of the term 
$$
\left( \fint_{H^{(K-N')}} w_\mu^r \right)^{\frac{1}{r}},
$$
we consider two sub-cases. 

\begin{enumerate}
    \item [$\bullet$] If $1 \le K-N' \le \alpha-1$, then as in \eqref{20200901eq12}, we have 
    $$
\fint_{H^{(K-N')}} w_\mu^r \le C_2 \cdot  a^{(K-N')r} .
$$
Moreover, we notice that in this case, we have
$$
0 \le K-\alpha \le N'-1 \le N,
$$
and hence
$$
\frac{1}{b^{K-\alpha}} \simeq 1,
$$
where the implicit constant above only depends on $b$, $p$ and $q$. Therefore, 
$$
\left( \fint_{H^{(K-N')}} w_\mu^r \right)^{\frac{1}{r}} \cdot  \frac{|J|}{\mu(J^*)} \lesssim \frac{1}{a^{N'}} \cdot \frac{1}{b^{K-\alpha}} \lesssim \frac{1}{a^{N}},
$$
which again implies the desired estimate \eqref{202009001eq11};

\medskip

\item [$\bullet$] If $\alpha \le K-N' \le 2\alpha$, then using the computation in \textit{Case I} of Proposition \ref{20200831prop02}, we have
$$
\fint_{H^{(K-N')}} w_\mu^r \le C_1 \cdot a^{\alpha r} b^{(K-N'-\alpha)r} .
$$
Therefore, 
$$
\left( \fint_{H^{(K-N')}} w_\mu^r \right)^{\frac{1}{r}} \cdot  \frac{|J|}{\mu(J^*)}  \lesssim b^{-N'} \le 1.
$$
This implies the desired estimate \eqref{202009001eq11} for this case. 
\end{enumerate}

As a conclusion, we have shown that
$$
\Theta_3 \lesssim 1,
$$
where the implicit constant above only depends on $a$, $b$, $p$ and $q$, independent of $\alpha$ and $K$. The proof is complete. 
\end{proof}

\begin{cor} \label{20200901cor01}
For any $r>1$, $RH_r^p \cap RH_r^q \neq RH_r$. In particular, $RH_1^p \cap RH_1^q \neq RH_1$.
\end{cor}

\begin{proof}
Note that from the proof of Proposition \ref{20200831prop01}, we indeed have if $w_\mu \in RH_r^q$ for some $r$ with $1<r<\frac{\ln q}{\ln b}$, then for the same choice of $r$, there holds $w_\mu \in RH_r^p$. 

To proof the desired claim, it suffices to take $0<a<1<b$ satisfying 
$$
1-a, b-1 \ll 1 \quad \textrm{and} \quad (q-1)a+b=q. 
$$
This suggests that we can make the ratio $\frac{\ln q}{\ln b}$ arbitrarily large. The desired claim then follows easily from Proposition \ref{20200831prop02} and Proposition \ref{20200831prop01}.
\end{proof}

We can also extend the above result to any finite family of primes, which immediately implies Theorem \ref{application theorem}.

\begin{cor}
\label{RH cor}
For any $r>1$ and $\{p_1, \dots, p_M\}$ any finite collection of primes, there holds that 
$$
\bigcap_{i=1}^{M} RH_r^{p_i} \neq RH_r.
$$
In particular,
$$
\bigcap_{i=1}^{M} RH_1^{p_i} \neq RH_1.
$$
\end{cor}

\begin{proof}
The proof of this result follows from an easy modification of the exhaustion procedure for finite families of primes, Theorem \ref{20200827thm01}, Proposition \ref{20200827prop01} and Corollary \ref{20200901cor01}, and we would like to leave the details to the interested reader. 
\end{proof}
 
Finally, we can prove analogous statements about the Muckenhoupt $A_p$ weights (which we will call $A_r$ weights to prevent confusion with $p$ being used for a prime).  Recall the definition:
\begin{defn}
Let $1<r<\infty$,  we say a weight $w\in A_r$ if 
\[
\sup_I \left(\fint_I w(x)dx\right)\left(\fint_I w(x)^{\frac{-1}{r-1}}dx\right)^{r-1} < \infty,
\]
where the supremum is taken over all intervals $I$. Moreover, we say $w \in A_\infty$ if $w \in A_r$ for some $r>1$, that is,
$$
A_\infty:=\bigcup_{r>1} A_r. 
$$
\end{defn}
 We define the \emph{$p$-adic} $A_r^p$, as well as \emph{$p$-adic} $A^p_\infty$, similarly by only allowing averages along $p$-adic intervals.  Note that the $A_r$ condition implies doubling.  An easy modification of the proof of Proposition \ref{20200831prop02}, Proposition \ref{20200831prop01} and Corollary \ref{20200901cor01} allows us to conclude the following analog for  Muckenhoupt $A_p$ weights, which immediately implies Theorem \ref{Ap thm}:
 
\begin{cor} \label{Ap cor}
For any $r>1$ and $\{p_1, \dots, p_M\}$ any finite collection of primes, then there holds that 
$$
\bigcap_{i=1}^{M} A_r^{p_i} \neq A_r.
$$
In particular,
$$
\bigcap_{i=1}^{M} A_\infty^{p_i} \neq A_\infty.
$$
 \end{cor}

\begin{proof}
Note that since any reverse H\"older weight (of class $r$) is also an $A_{r_1}$ weight for some $1 <r_1<\infty$, and similarly for the prime classes, Corollary \ref{RH cor} indeed directly implies Corollary \ref{Ap cor} holds for \emph{some $r>1$}. 

However, it turns out that we can improve such a range to all $r>1$ and the proof is parallel to those in  Proposition \ref{20200831prop02}, Proposition \ref{20200831prop01} and Corollary \ref{20200901cor01}. Let us mention the necessary modifications that we need to make to prove the result: recall that at the beginning of Proposition  \ref{20200831prop02}, we pick $r>1$, such that
$$
1<r<\frac{\ln q}{\ln b}. 
$$
Now we replace this by
$$
1-\frac{\ln a}{\ln q}<r<\infty.
$$
Note that since $0<a<1$, the term on the left hand side above is strictly bigger than $1$, in particular, this means that when proving the analog of Corollary \ref{20200901cor01} for Muckenhoupt weights, we again choose $a$ sufficiently close to $1$, to make $1-\frac{\ln a}{\ln q}$ arbitrarily close to the threshold $1$.  The rest of the proof then follows by interchanging the role of $b$ and $a$, and replacing the role of $r$ in the proof of  Proposition \ref{20200831prop02}, Proposition \ref{20200831prop01} and Corollary \ref{20200901cor01} by $-\frac{1}{r-1}$, for example, we may define
$$
B_3:=\frac{a^{-\frac{1}{r-1}}}{q} \quad \textrm{and} \quad B_4:=\frac{b^{-\frac{1}{r-1}}}{q}, 
$$
to replace $B_1$ and $B_2$ there, respectively. We would like to leave the detail to the interested reader. 
\end{proof}

\end{document}